\documentclass[sn-mathphys]{sn-jnl}% Math and Physical Sciences Reference Style
\usepackage{amssymb,amsmath,mathrsfs, amsthm}

\usepackage{cases}

\jyear{2021}%

\theoremstyle{thmstyleone}%
\numberwithin{equation}{section}
\newtheorem{theorem}{Theorem}[section]%  meant for continuous numbers
\newtheorem{proposition}[theorem]{Proposition}%
\newtheorem{corollary}{Corollary}[section]
\theoremstyle{thmstyletwo}%
\newtheorem{remark}[theorem]{Remark}%
\newtheorem{lemma}{Lemma}[section]

\theoremstyle{thmstylethree}%
\newtheorem{definition}[theorem]{Definition}%

\begin{document}

\title[Article Title]{Characterization of strongly convex K\"ahler-Berwald metrics
}


\author{\fnm{Wei} \sur{Xiao}}\email{wxiaomath@126.com}

\author{\fnm{Chunping} \sur{Zhong}\footnote{Corresponding author: zcp@xmu.edu.cn}}\email{zcp@xmu.edu.cn}


\affil{\orgdiv{School of Mathematical Sciences}, \orgname{Xiamen
University},  \city{Xiamen}, \postcode{361005}, \country{China}}

%\affil[2]{\orgdiv{Department}, \orgname{Organization}, \orgaddress{\street{Street}, \city{City}, \postcode{10587}, \state{State}, \country{Country}}}

%\affil[3]{\orgdiv{Department}, \orgname{Organization}, \orgaddress{\street{Street}, \city{City}, \postcode{610101}, \state{State}, \country{Country}}}


\abstract{
Let $F: T^{1,0}M\rightarrow[0,+\infty)$ be a strongly convex complex Finsler metric on a complex manifold $M$ and $\pmb{J}$ the canonical complex structure on the complex manifold $T^{1,0}M$. We give a geometric characterization of strongly convex K\"ahler-Berwald metrics. In particular, we prove that $\pmb{J}$  is horizontally parallel with respect to the Cartan connection  iff $F$ is a K\"ahler-Berwald metric. We also prove that the Cartan connection and the Chern-Finsler connection associated to $F$ coincide iff $\pmb{J}$ is both horizontal and vertical parallel with respect to the Cartan connection. Based on these results, we give a rigidity theorem of strongly convex K\"ahler-Berwald metrics with constant holomorphic sectional curvatures.}




\maketitle

\section{Introduction and main results}

The search for invariant metrics in complex Finsler geometry has revealed a fundamental dichotomy. Zhong \cite{Zhong-b} proved that no $\operatorname{Aut}(B_n)$-invariant strongly pseudoconvex complex Finsler metric exists on the open unit ball $B_n \subset \mathbb{C}^n$ except for constant multiples of the Bergman metric. However, on the unit polydisk $P_n$ ($n \geq 2$), there exist infinitely many  non-Hermitian quadratic $\operatorname{Aut}(P_n)$-invariant K\"ahler-Finsler metrics.

Subsequent work has generalized these findings. Lin and Zhong \cite{Lin-Zhong} characterized all such invariant metrics which are strongly pseudconvex on the reducible bounded symmetric domain $P_n$, showing that they are all K\"ahler-Berwald metrics. For irreducible bounded symmetric domains $\mathfrak{R}_A$  with $\operatorname{rank}(\mathfrak{R}_A) \geq 2$, Ge and Zhong \cite{Ge-Zhong} also constructed infinitely many $\operatorname{Aut}(\mathfrak{R}_A)$-invariant K\"ahler-Berwald metrics. Moreover, Zhong \cite{Zhong-c} characterized all holomorphically invariant metrics which are strongly pseudoconvex on the classical domains, proving that they are all K\"ahler-Berwald metrics and they share key holomorphic sectional curvature properties with the Bergman metric. These results underscore the central role of K\"ahler-Berwald manifolds in complex Finsler geometry.

The purpose of this  paper is to give a geometric characterization of strongly convex K\"ahler-Berwald metrics. Given a strongly convex complex Finsler metric $F: T^{1,0}M \rightarrow [0, +\infty)$ on a complex manifold $M$, with the canonical complex structures $J$ (on $M$) and $\pmb{J}$ (on $T^{1,0}M$). Our method is via investigating different levels of parallelism of the canonical complex structure $\pmb{J}$ with respect to the Cartan connection $\nabla$ associated to $F$. We find that such parallelism exhibits strong rigidity, which leads to the study of a special class of complex Finsler metrics, namely the so-called K\"ahler-Berwald metrics.

 Let's firstly take a look of the Hermitian case. A Hermitian metric on a complex manifold $M$ is a Riemannian metric $g$ on the underling smooth manifold that is $J$-invariant:
  $
  g(JX,JY)=g(X,Y)
  $
  for all real tangent vectors $X$ and $Y$. Locally $g$ can be represented  by
 $g=2\text{Re}(g_{\alpha\bar{\beta}}dz^\alpha\otimes d\bar{z}^\beta)$ with the associated real $(1,1)$-form $\omega=ig_{\alpha\bar{\beta}}dz^\alpha\wedge d\bar{z}^\beta$. Denote $\nabla$ the Levi-Civita connection of $g$. In \cite{Mok},  N. Mok
 gave the following geometric definition of a Hermitian manifold to be a K\"ahler manifold.
\begin{definition}\cite{Mok}\label{MD}
A Hermitian manifold $(M,g)$ is said to be K\"ahler iff the types of complexified tangent vector fields are preserved under parallel transport.
\end{definition}

Moreover, N. Mok proved the following
\begin{proposition}[\cite{Mok}]\label{prop-I}
Let $(M,g)$ be a Hermitian manifold such that $g$ is given by $2\text{Re}\left(g_{\alpha\bar{\beta}}dz^\alpha\otimes d\bar{z}^\beta\right)$ in local holomorphic coordinates $(z^\alpha)$. Then, $(M,g)$ is K\"ahler iff one of the following equivalent conditions is satisfied:
\begin{enumerate}
\item[(1)] types of complexified tangent vectors are preserved under parallel transport;

\item[(2)] for any parallel (real) vector field $\eta$ along a smooth curve $\gamma$, $J\eta$ is also parallel;

\item[(3)] $\nabla J\equiv 0$, i.e., the almost complex structure $J$ is parallel;

\item[(4)] $\nabla\omega\equiv 0$, i.e., the Hermitian form $\omega$ is parallel;

\item[(5)] $d\omega\equiv 0$, i.e., the Hermitian form $\omega$ is closed;

\item[(6)] locally there exists a potential function $\varphi$ such that $g_{\alpha\bar{\beta}}=\frac{\partial^2\varphi}{\partial z^\alpha\partial \bar{z}^\beta}$;

\item[(7)] at every point $p\in M$ there exists complex geodesic coordinates $(z^\alpha)$ in the sense that the Hermitian metric $g$ is represented by the Hermitian matrix $(g_{\alpha\bar{\beta}})$  satisfying $g_{\alpha\bar{\beta}}(p)=\delta_{\alpha\beta}$ and $dg_{\alpha\bar{\beta}}(p)=0$.
\end{enumerate}
\end{proposition}

The study of K\"ahler conditions in complex Finsler geometry originated with Royden's work \cite{Royden}. Let $\triangle(r)$ denote the open disk of radius $r$ in $\mathbb{C}$ centered at the origin which is endowed with the Poincar\'e metric. Under the assumption that for any $(z;v) \in T^{1,0}M$ there exists a holomorphic map $\varphi: \triangle(r) \rightarrow M$ with $\varphi(0)=z$, $\varphi_*(0)=v$, and such that the curve $\gamma(t) = \varphi(e^{i\theta}t)$ is a geodesic tangent to $\mathbb{C} \cdot v$ at $z$ for each $\theta \in \mathbb{R}$ (a condition later termed \emph{Royden's condition} in \cite{Aikou-b}), Royden proved that $F$ must be a weakly K\"ahler-Finsler metric along such maps $\varphi$. For a clear reformulation as well as a characterization of K\"ahler-Finsler metrics, we refer to Theorem 2.3.10 in \cite{AP}.

In their work \cite{chen2}, Chen, Liu, and Zhao investigated the weak parallelism of the canonical complex structure $J$ on a complex manifold $M$ with respect to the real Berwald connection $\breve{\nabla}$ of a strongly convex weakly K\"ahler-Finsler metric $F$. They proved that the real fundamental tensor of $F$ is $J$-invariant iff $F$ comes from a Hermitian metric. Furthermore, they investigated the relationship between $J$ and both the real and complex Ricci curvatures of $F$, they also showed that for strongly convex K\"ahler-Berwald metrics, the real and complex notions of an Einstein metric coincide.

A further inspiration for this work comes from the classical study by M. Abate and G. Patrizio on the Cartan and Chern-Finsler connections for a strongly convex complex Finsler metric \cite{AP}. In particular, they inquire about the existence of a canonical connection on real Finsler manifolds that agrees with the Chern-Finsler connection in the K\"ahler-Finsler case (see \cite{AP}, p. 122). Moreover, they observe that the Cartan connection includes elements deemed geometrically non-essential, since these do not contribute to the first and second variation formulas (see \cite{AP}, p. 125).

It is crucial to distinguish the base manifold's canonical complex structure $J$ (a $(1,1)$-tensor on $M$) from the complex structure $\pmb{J}$ (a $(1,1)$-tensor on $T^{1,0}M$), on which the geometry of a strongly convex complex Finsler metric $F$ fundamentally resides. In this paper, we derive rigidity results by investigating the parallelism of $\pmb{J}$ with respect to the associated Cartan connection $\nabla$, an approach grounded in the concept of parallelism relative to a real Finsler connection of $F$.
 Our first main result is
\begin{theorem}\label{mth-a}
Let $F: T^{1,0}M\rightarrow [0,+\infty)$ be a strongly convex complex Finsler metric on a complex manifold $M$.  Then $\nabla \pmb{J}=0$ iff the Cartan connection $\nabla$ associated to $F$ coincides with the Chern-Finsler connection $D$ associated to $F$.
\end{theorem}
%\begin{remark}
%Theorem \ref{mth-a} actually gives a rigidity result in complex Finsler geometry since $\nabla \pmb{J}=0$ implies that $F$ must be a strongly convex K\"ahler-Berwald metric. See the following Theorem \ref{mth-c}.
%\end{remark}

In Hermitian geometry, if $h$ is a Hermitian metric on a complex manifold $M$ and $g=2\text{Re}\,h$ is the associated Riemannian metric, and if the canonical complex structure $J$ on $M$ is parallel with respect to the Levi-Civita connection of $g$, then $h$ must be a K\"ahler metric. Theorem \ref{mth-a} generalizes this classical result to the Finsler setting.
%Note that the Cartan connection in Theorem \ref{mth-a} is in the sense of \cite{AP}, which is different from the Cartan connection defined in \cite{Aikou-c}.

Our second main result is
  \begin{theorem}\label{mth-b}
Let $F: T^{1,0}M\rightarrow [0,+\infty)$ be a strongly convex weakly K\"ahler-Finsler metric on a complex manifold $M$ and $\sigma:[0,1]\rightarrow M$ a smooth regular curve in $M$. Then the following assertions are equivalent:
\begin{enumerate}
\item[(1)]  $F$ is a strongly convex K\"ahler-Berwald metric;

\item[(2)] The types of complexified  vectors in $T_{\mathbb{R}}M$ are preserved under parallel transport along $\sigma$ with respect to  $\nabla$;

\item[(3)] For any parallel real vector field $V$ along $\sigma$ with respect to  $\nabla$,  $ JV$ is also parallel along $\sigma$  with respect to $\nabla$;

\item[(4)] $\nabla_X  \pmb{J}\equiv 0$ for all $X\in\mathcal{H}_{\mathbb{R}}$;

\end{enumerate}
\end{theorem}

\begin{remark}
There are lots of nontrivial (non-Hermitian quadratic) strongly convex K\"ahler-Berwald metrics \cite{Zhong-b,Lin-Zhong,Zhong-c}. So Theorem \ref{mth-b} gives a geometric characterization of  strongly convex K\"ahler-Berwald manifolds, which partially generalizing Proposition \ref{prop-I} to Finsler setting.
\end{remark}

As an application of Theorem \ref{mth-a}, we obtain the following

\begin{theorem}\label{mth-c}
Let $F: T^{1,0}M\rightarrow [0,+\infty)$ be a strongly convex complex Finsler metric on a complex manifold $M$.  If $\nabla \pmb{J}=0$ and $F$ has constant holomorphic sectional curvature $c\neq 0$, then $F$ is necessary a K\"ahler-Einstein metric on $M$.
\end{theorem}

Using Theorem \ref{mth-c}, we immediately obtain the following
 \begin{theorem}\label{mth-d}
 	Let $F: T^{1,0}M\rightarrow[0,+\infty)$ be a complete strongly convex complex Finsler metric on a simply connected complex manifold $M$.   If $\nabla\pmb{J}=0$ and $F$ has constant holomorphic sectional curvature $c\in \mathbb{R}$. Then
 \begin{itemize}
 \item[(i)] if $c<0$, $(M,F)$ is a K\"ahler manifold which is holomorphically isometric to the open unit ball $B_n$ in $\mathbb{C}^n$ with a constant multiple of the Bergman metric, namely
 \begin{equation}
 F^2(z;v)=-\frac{4}{c}\frac{(1-\|z\|^2)\|v\|^2+\vert\langle z,v\rangle\vert^2}{(1-\|z\|^2)^2}
 \end{equation}

\item[(ii)] if $c=0$, $(M,F)$ is locally a complex Minkowski space which is locally holomorphic isometric to $\mathbb{C}^n$ with some complex Minkowski metric $F(z;v)=f(v)$;
 	
 \item[(iii)] if $c>0$,  $(M,F)$ is a K\"ahler manifold which is holomorphically isometric to the complex projective space $\mathbb{CP}^n$ with a constant multiple of the Fubini-Study metric, namely
 \begin{equation}
 F^2(z;v)=\frac{4}{c}\frac{(1+\|z\|^2)\|v\|^2-\vert\langle z,v\rangle\vert^2}{(1+\|z\|^2)^2}.
 \end{equation}
\end{itemize}
\end{theorem}

In \cite{Aikou-b}, under the assumptions that $(M,F)$ is a simply connected and complete complex manifold modeled on a complex Minkowski space $(\mathbb{C}^n,f)$, that $F$ is K\"ahler-Finsler and satisfies the Royden condition, and that $(M,F)$ has constant holomorphic sectional curvature $c$, Aikou \cite{Aikou-b} actually outlined the proof of the above Theorem \ref{mth-c}. In contrast, our hypothesis $\nabla\pmb{J}=0$ simple and purely a geometric condition. Furthermore, our proof of Theorem \ref{mth-d} is  mainly based on the rigidity Theorem \ref{RKE}, which is essentially different from the approach of Aikou \cite{Aikou-b}.

\section{Preliminary}

\subsection{Real and holomorphic tangent bundles of a complex manifold}
This section recalls some necessary notations and definitions, which can be found in \cite{AP,BCS}.

Let $M$ be a complex manifold of complex dimension $n (\geq 2)$, and let $\pi:T_{\mathbb{R}}M\rightarrow M$ be its real tangent bundle. Denote by $J:T_{\mathbb{R}}M\rightarrow T_{\mathbb{R}}M$ the canonical complex structure on $M$, and throughout this paper we denote by $i=\sqrt{-1}$ the imaginary unit. We also use Einstein summation convention,  lowercase greek indices will run from $1$ to $n$, whereas lowercase roman indices will run from $1$ to $2n$, and $\alpha^\ast=\alpha+n,\beta^\ast=\beta+n,$ etc.

Let $z=(z^1,\cdots,z^n)$ be the local holomorphic coordinates on an open subset $U\subset M$. Setting $z^\alpha=x^\alpha+ix^{\alpha^\ast}$ for $\alpha=1,\cdots,n$, we obtain the corresponding real coordinates $x=(x^1,\cdots,x^{2n})$ on $U$. The canonical complex structure $J$ on $M$ satisfies
\begin{equation*}
	J\left(\frac{\partial}{\partial x^{\alpha}}\right) =\frac{\partial }{\partial x^{\alpha^\ast}},\quad J\left(\frac{\partial}{\partial x^{\alpha^\ast}}\right)=-\frac{\partial}{\partial x^\alpha},\quad \alpha=1,\cdots,n.
\end{equation*}

Extending $J$ by complex linearity to the complexified tangent bundle $T_{\mathbb{C}}M$,  we
obtain the decomposition
$T_{\mathbb{C}}M = T^{1,0}M \oplus T^{0,1}M$, where $T^{1,0}M$ and $T^{0,1}M$ are the holomorphic and anti-holomorphic tangent bundle, respectively.  Local frames for $T^{1,0}
M$  and $T^{0,1}M$ over $U$ are given by
$$\left\{\frac{\partial}{\partial z^1},\cdots,\frac{\partial}{\partial z^n}\right\}\quad \text{and}\quad  \left\{\frac{\partial}{\partial \bar{z}^1},\cdots,\frac{\partial}{\partial \bar{z}^n}\right\}$$
respectively, where
$$\frac{\partial}{\partial z^{\alpha}}=\frac{1}{2}\left(\frac{\partial}{\partial x^{\alpha}}-i\frac{\partial}{\partial x^{\alpha^\ast   }}\right),\quad
\frac{\partial}{\partial \bar{z}^{\alpha}}=\frac{1}{2}\left(\frac{\partial}{\partial x^{\alpha}}+i\frac{\partial}{\partial x^{\alpha^\ast   }}\right).$$

There exists an explicit $\mathbb{R}$-isomorphism between  $T_{\mathbb{R}}M$ and $T^{1,0}M$, defined by
$$_o:T_{\mathbb{R}}M \rightarrow T^{1,0}M, \;y\mapsto v=\frac{1}{2}(y-iJy),\quad \forall y=y^j\frac{\partial}{\partial x^j}$$
and
$$^o:T^{1,0}M \rightarrow T_{\mathbb{R}}M, \;v\mapsto y=v+\bar{v},\quad \forall v=v^\alpha\frac{\partial}{\partial z^\alpha}.$$
It is easy to verify that $y_o=v$ and $v^o=y$ if we set $v^\alpha=y^\alpha+iy^{\alpha^\ast}$ for $\alpha=1,\cdots,n$.

A vector $v\in T_p^{1,0}M$ can be written as $v=v^{\alpha}\frac{\partial}{\partial z^{\alpha}}\vert_p$. This induces holomorphic coordinates $(z;v)=(z^1,\cdots,z^n;v^1,\cdots,v^n)$ on $\pi^{-1}(U)\subset T^{1,0}M$, showing that $T^{1,0}M$ is a complex manifold of complex dimension $2n$. Setting $v^{\alpha}=y^{\alpha}+\sqrt{-1}y^{\alpha^\ast   }$ for $\alpha=1,\cdots,n$, we define
\begin{equation}
	\frac{\partial }{\partial v^{\alpha}}=\frac{1}{2}\left(\frac{\partial}{\partial y^{\alpha}}-i\frac{\partial}{\partial y^{\alpha^\ast   }}\right),\quad\frac{\partial }{\partial \bar{v}^{\alpha}}=\frac{1}{2}\left(\frac{\partial}{\partial y^{\alpha}}+i\frac{\partial}{\partial y^{\alpha^\ast   }}\right).
\end{equation}

In the following, we denote by $\pmb{J}$ the canonical complex structure on the complex manifold $T^{1,0}M$, namely
\begin{equation}\label{CS-T}
	\begin{aligned}
		&&\pmb{J}\left(\frac{\partial}{\partial x^\alpha}\right)=\frac{\partial}{\partial x^{\alpha^\ast   }},\quad \pmb{J}\left(\frac{\partial}{\partial x^{\alpha^\ast   }}\right)=-\frac{\partial}{\partial x^{\alpha}},\\
	&&	 \pmb{J}\left(\frac{\partial}{\partial y^\alpha}\right)=\frac{\partial}{\partial y^{\alpha^\ast }},\quad \pmb{J}\left(\frac{\partial}{\partial y^{\alpha^\ast   }}\right)=-\frac{\partial}{\partial y^\alpha},
	\end{aligned}
\end{equation}
for $\alpha=1,\cdots,n$.

The real (resp. complex) Finsler metrics considered in this paper are smooth outside the zero section of the real (resp. holomorphic) tangent bundle of $M$. In the following, we denote by $\tilde{M}$ either $T_{\mathbb{R}}M\setminus\{0\}$ or $T^{1,0}M\setminus\{0\}$ when no confusion arises.
A local real frame for $T_{\mathbb{R}}\tilde{M}$ is given by
$$\left\{\frac{\partial}{\partial x^1},\cdots,\frac{\partial}{\partial x^{2n}},\frac{\partial}{\partial y^1},\cdots,\frac{\partial}{\partial y^{2n}}\right\};$$
Correspondingly, a local holomorphic frame for $T^{1,0}\tilde{M}$ is
$$\left\{\frac{\partial}{\partial z^1},\cdots,\frac{\partial}{\partial z^n},\frac{\partial}{\partial v^1},\cdots,\frac{\partial}{\partial v^n}\right\}.$$

Similar to $T_{\mathbb{R}}M$ and $T^{1,0}M$, there is a natural $\mathbb{R}$-isomorphism between $T_{\mathbb{R}}\tilde{M}$ and $T^{1,0}\tilde{M}$, defined by
$$_o:T_{\mathbb{R}}\tilde{M}\rightarrow T^{1,0}\tilde{M},\quad U\mapsto U_o=\frac{1}{2}(U-i\pmb{J}U),\quad\forall U\in T_{\mathbb{R}}\tilde{M}$$
and
 $$^o:T^{1,0}\tilde{M}\rightarrow T_{\mathbb{R}}\tilde{M},\quad X\mapsto X^o=X+\overline{X},\quad\forall X\in T^{1,0}\tilde{M}.$$

It is easy to check that
\begin{equation*}
\frac{\partial}{\partial z^\alpha}=\left(\frac{\partial}{\partial x^\alpha}\right)_o,\quad i\frac{\partial}{\partial z^\alpha}=\left(\frac{\partial}{\partial x^{\alpha^\ast   }}\right)_o,\quad\frac{\partial}{\partial v^\alpha}=\left(\frac{\partial}{\partial y^\alpha}\right)_o,\quad i\frac{\partial}{\partial v^\alpha}=\left(\frac{\partial}{\partial y^{\alpha^\ast   }}\right)_o
\end{equation*}
or equivalently,
\begin{equation*}
\left(\frac{\partial}{\partial z^\alpha}\right)^o=\frac{\partial}{\partial x^\alpha},\quad \left(i\frac{\partial}{\partial z^\alpha}\right)^o=\frac{\partial}{\partial x^{\alpha^\ast   }},\quad\left(\frac{\partial}{\partial v^\alpha}\right)^o=\frac{\partial}{\partial y^\alpha},\quad \left(i\frac{\partial}{\partial v^\alpha}\right)^o=\frac{\partial}{\partial y^{\alpha^\ast }}.
\end{equation*}

\subsection{Strongly convex complex Finsler metrics}
\begin{definition}\cite{BCS}
A real (strongly convex) Finsler metric on a smooth manifold $M$ of real dimension $n$ is a function  $F:T_{\mathbb{R}}M\rightarrow [0,+\infty)$ satisfying
\begin{enumerate}

\item[(i)]  $F$ is positive and $C^\infty$ on  $\tilde{M}=T_{\mathbb{R}}M\setminus\{0\}$;

\item[(ii)] $F(x;\lambda y)=\lambda  F(x;y)$ for all $(x;y)\in T_{\mathbb{R}}M$ and $\lambda>0$;

\item[(iii)] The real fundamental tensor
$$
(g_{jk}):=\left(\frac{1}{2}\frac{\partial^2F^2}{\partial y^j\partial y^k}\right)
$$
is positive definite on $\tilde{M}$.
\end{enumerate}
\end{definition}

\begin{definition}\cite{AP}
A strongly pseudoconvex complex Finsler metric on a complex manifold $M$ of complex dimension $n$ is a function  $F:T^{1,0}M\rightarrow [0,+\infty)$ satisfying
\begin{enumerate}
\item[(i)] $F$ is positive and $C^\infty$ on $\tilde{M}=T^{1,0}M\setminus\{0\}$;

\item[(ii)] $F(z;\lambda v)=\vert \lambda\vert  F(z;v)$ for all $(z;v)\in T^{1,0}M$ and $\lambda\in\mathbb{C}$;

\item[(iii)] The complex fundamental tensor
$$
(G_{\alpha\bar{\beta}}):=\left(\frac{\partial^2F^2}{\partial v^\alpha \partial\bar{v}^\beta}\right)
$$
is Hermitian positive definite on $\tilde{M}$.
\end{enumerate}
\end{definition}

Let $F:T^{1,0}M\rightarrow [0,+\infty)$ be a strongly convex complex Finsler metric. It induces a real Finsler metric $F_{\mathbb{R}}$ on $M$ via
\begin{equation}
F_{\mathbb{R}}(y):=F(y_o)=F(v).\label{RaC}
\end{equation}
Clearly, $F_{\mathbb{R}}$ is $J$-invariant, namely $F_{\mathbb{R}}(Jy)=F_{\mathbb{R}}(y)$.\label{J-m}

 Because of equation \eqref{RaC}, for simplicity, we will make no distinction between notations $F$ and $F_{\mathbb{R}}$ in the following if it causes no confusion.

\subsection{Real geodesic}

Let $F:T^{1,0}M\rightarrow [0,+\infty)$ be a strongly convex complex Finsler metric on a complex manifold $M$, and $\sigma:[a,b]\rightarrow M$ be a smooth curve. According to the proof of Corollary 1.5.2 in \cite{AP},  $\sigma$ is a real geodesic for $F$ iff it satisfies the following differential equation
\begin{equation}
\ddot{\sigma}^k+\hat{\Gamma}_j^k(\sigma;\dot{\sigma})\dot{\sigma}^j=0,\label{geodesic-1}
\end{equation}
where the coefficients $\hat{\Gamma}_j^k$ are the Christoffel symbols of the real non-linear connection $\tilde{\nabla}:\mathcal{X}(T_{\mathbb{R}}M)\rightarrow \mathcal{X}(T_{\mathbb{R}}^\ast M\otimes T_{\mathbb{R}}M)$ associated to $F$. More precisely,
\begin{equation}
\hat{\Gamma}_j^k(x;y):=\frac{\partial \hat{\mathbb{G}}^k}{\partial y^j},\label{geodesic-2}
\end{equation}
where
\begin{equation}
\hat{\mathbb{G}}^k:=\frac{1}{4}g^{kl}\left(\frac{\partial^2F^2}{\partial x^a\partial y^l}y^a-\frac{\partial F^2}{\partial x^l}\right)\label{geodesic-3}
\end{equation}
are called the real geodesic spray coefficients of $F$.

 Equations \eqref{geodesic-1}-\eqref{geodesic-3} show that the real non-linear connection coefficients $\hat{\Gamma}_j^k$ completely determine the real geodesics of $F$.  The real non-linear connection $\tilde{\nabla}$ induces a direct sum decomposition of the tangent bundle of the slit tangent bundle $\tilde{M}$:
 $$T_{\mathbb{R}}\tilde{M}=\mathcal{H}_{\mathbb{R}}\oplus\mathcal{V}_{\mathbb{R}},$$
 where $\mathcal{V}_{\mathbb{R}}=\ker d\pi$ is the real vertical bundle, and $\mathcal{H}_{\mathbb{R}}$ is the real horizontal bundle determined by $\tilde{\nabla}$. This decomposition is locally characterized by adapted frames
 $$\mathcal{H}_{\mathbb{R}}=\text{span}\left\{\frac{\delta}{\delta x^1},\cdots,\frac{\delta}{\delta x^{2n}}\right\},\quad
\mathcal{V}_{\mathbb{R}}=\text{span}\left\{\frac{\partial}{\partial y^1},\cdots,\frac{\partial}{\partial y^{2n}}\right\},
$$
where
$\frac{\delta}{\delta x^j}:=\frac{\partial}{\partial x^j}-\hat{\Gamma}_j^k\frac{\partial}{\partial y^k}$ for $j,k=1,\cdots,2n$. The dual bundles $\mathcal{H}_{\mathbb{R}}^\ast$ and $\mathcal{V}_{\mathbb{R}}^\ast$ are correspondingly spanned by  $\{dx^1,\cdots,dx^{2n}\}$ and $\{\delta y^1,\cdots,\delta y^{2n}\}$, respectively, with $\delta y^j:=dy^j+\hat{\Gamma}^j_k dx^k$.

%There is a canonical isomorphism $\hat{\iota}_y = d(j_{\pi(y)})_y \circ k_y$ from $(T_{\mathbb{R}}M)_p$ to $(\mathcal{V}_{\mathbb{R}})y$, where $j_p:(T_{\mathbb{R}}M)_p\rightarrow T_{\mathbb{R}}M$ is the inclusion and $k_y$ is the canonical identification of $(T_{\mathbb{R}}M)_p$ with its tangent space at $y\in(T_{\mathbb{R}}M)_p$. The real radial vertical field is defined as $\hat{\iota}(y) = \hat{\iota}_y(y)$.

%A real horizontal map $\hat{\Theta}: \mathcal{V}_{\mathbb{R}} \rightarrow T_{\mathbb{R}}\tilde{M}$ is induced by the nonlinear connection $\tilde{D}$, providing an isomorphism onto $\mathcal{H}_{\mathbb{R}}$. The horizontal lift of a vector field $X \in \mathcal{X}(T_{\mathbb{R}}M)$ is the unique section $X^H$ of $\mathcal{H}_{\mathbb{R}}$ satisfying $d\pi(X^H) = X$.

%Composing these maps, we define $\hat{\chi}_y := \hat{\Theta}_y \circ \hat{\iota}_y$, which maps $(T_{\mathbb{R}}M)_p$ to $(\mathcal{H}_{\mathbb{R}})_y$. The real horizontal radial vector field is then given by $\hat{\chi} = \hat{\Theta} \circ \hat{\iota} \in \mathcal{X}(\mathcal{H}_{\mathbb{R}})$.

%In local coordinates, for $u=u^j\frac{\partial}{\partial x^j}$ in $(T_{\mathbb{R}}M)_p$, we have
%$$
%\hat{\iota}_y(u)=u^j\frac{\partial}{\partial y^j}\Big\vert _y,\quad\hat{\iota}(y)=y^j\frac{\partial}{\partial y^j}\Big\vert _y,
%$$
%and
% $$\hat{\chi}_y(u)=u^j\left(\frac{\partial}{\partial x^j}-\hat{\Gamma}^k_j(x;y)\frac{\partial}{\partial y^k}\right)\quad\text{and}\quad \hat{\chi}(y)=y^j\left(\frac{\partial}{\partial x^j}-\hat{\Gamma}^k_j(x;y)\frac{\partial}{\partial y^k}\right).
% $$

\subsection{The Cartan connection}

Let $F:T^{1,0}M\rightarrow [0,+\infty)$ be a strongly convex complex Finsler metric on a complex manifold $M$, and let $g=\langle\cdot\vert \cdot\rangle$ be the natural Riemannian metric induced by $F$ on the vertical bundle $\mathcal{V}_{\mathbb{R}}$. In local coordinates,  this metric is expressed as
\begin{equation}
g=g_{jk}(x;y)\delta y^j\otimes \delta y^k
\end{equation}
so that for any vertical vector fields $X=X^j\frac{\partial}{\partial y^j}$ and $Y=Y^k\frac{\partial}{\partial y^k}$ in  $(\mathcal{V}_{\mathbb{R}})_{(x;y)}$,
\begin{equation}
\langle X\vert Y\rangle=g_{jk}(x;y)X^jY^k\quad\text{where}\quad g_{jk}(x;y)=g\left(\frac{\partial}{\partial y^j},\frac{\partial}{\partial y^k}\right).
\end{equation}

\begin{theorem}\label{cartan}\cite{AP}
	Let $F$ be a real Finsler metric on  $M$, and let $\langle\cdot\vert \cdot\rangle$ be the Riemannian structure on $\mathcal{V}_{\mathbb{R}}$ induced by $F$. Then there exists a unique real vertical connection
	$$\nabla : \mathcal{X}(\mathcal{V}_{\mathbb{R}}) \rightarrow \mathcal{X}(T_{\mathbb{R}}^*\tilde{M}\otimes\mathcal{V}_{\mathbb{R}})$$
satisfying the following properties:
	\begin{enumerate}
	\item[(i)] $\nabla$ is good;
	
	\item[(ii)] $\nabla$ is compatible with the metric: for any $X \in T_{\mathbb{R}}\tilde{M}$ and $V,W\in \mathcal{X}(\mathcal{V}_{\mathbb{R}})$,
	$$ X\langle V\vert W\rangle =\langle \nabla_X V\vert W\rangle+\langle V\vert \nabla_X W\rangle;$$
	
	\item[(iii)] The vertical torsion vanishes: $\theta(V,W)=0$ for all $V,W\in \mathcal{V}_{\mathbb{R}}$;
	
	\item[(iv)] The horizontal torsion is vertical: $\theta(H,K)\in \mathcal{V}_{\mathbb{R}} $ for all  $H,K\in \mathcal{H}_{\mathbb{R}}$.
\end{enumerate}
\end{theorem}

In local coordinates, the connection $1$-form $\hat{\omega}_j^{\;k}$ of the Cartan connection is given by
$$\hat{\omega}_j^{k}=\hat{\Gamma}_{j;l}^kdx^l+\hat{\Gamma}^k_{jl}\delta y^l,$$
where the coefficients are defined as
 \begin{eqnarray*}
\hat{\Gamma}_{j;l}^k=\frac{g^{ks}}{2}\left(\frac{\delta g_{sj}}{\delta x^l}-\frac{\delta g_{jl}}{\delta x^s}+\frac{\delta g_{ls}}{\delta x^j}\right),\quad
	\hat{\Gamma}^k_{jl}=\frac{1}{4}g^{ks}\frac{\partial^2F^2}{\partial y^j\partial y^l\partial y^s}.
\end{eqnarray*}
These coefficients are symmetric in their lower indices: $\hat{\Gamma}_{j;l}^k=\hat{\Gamma}_{l;j}^k$ and $\hat{\Gamma}_{jl}^k=\hat{\Gamma}_{lj}^k$. Moreover, since the Cartan connection is good, it follows from Lemma 1.2.2 in \cite{AP} that the non-linear connection coefficients satisfy
\begin{equation}
\hat{\Gamma}^k_j=\hat{\Gamma}^k_{l;j}y^l.\label{hcp}
\end{equation}

Using the horizontal Cartan connection coefficients, one can introduce the parallelism of a real vector field $V\in\mathcal{X}(T_{\mathbb{R}}M)$ along a smooth curve $\sigma:[0,1]\rightarrow M$. In real Finsler geometry, however, there are several different definitions for the parallel transport of a real vector field along a curve. Different definitions stem from distinct geometric interpretations of "parallelism" and the use of different connections. In general, these definitions are not equivalent on a real Finsler manifold, we refer to \cite{AP},\cite{BCS} for more details.

\begin{definition}\label{P}
 Let $V(t)=V^k(t)\frac{\partial}{\partial x^k}$ be a real vector field defined along a smooth curve $\sigma:[0,1]\rightarrow M$. We say $V$ is parallel along $\sigma$ with respect to the Cartan connection $\nabla$  if locally
 \begin{equation}
\frac{\pmb{D}V^k}{dt}:=\frac{dV^k}{dt}+V^l\hat{\Gamma}_{l;j}^k(\sigma(t);\dot{\sigma}(t))\dot{\sigma}^j(t)=0,\quad\forall k=1,\cdots,2n.\label{parallel}
\end{equation}
\end{definition}
In particular, $\sigma$ itself is a geodesic iff  $\dot{\sigma}$ is parallel along $\sigma$ with respect to $\nabla$, which is precisely the equation  \eqref{geodesic-1} since we always have $$\dot{\sigma}^l(t)\hat{\Gamma}_{l;j}^k(\sigma(t);\dot{\sigma}(t))=\hat{\Gamma}_{j}^k(\sigma(t);\dot{\sigma}(t)).$$

\begin{remark}
By the above definition, if  $V_1$ and $V_2$ are two real vector fields defined along $\sigma$, which are parallel along $\sigma$ with respect to the Cartan connection $\nabla$, then $aV_1+bV_2$ are also parallel along $\sigma$ with respect to $\nabla$ for any $a,b\in\mathbb{R}$.
\end{remark}

The torsion $\theta\in\mathcal{X}(\wedge^2 T_{\mathbb{R}}^\ast \tilde{M}\otimes T_{\mathbb{R}}\tilde{M})$ of $\nabla$ has the local expression:
\begin{eqnarray}
	\theta=\theta^{j}\otimes \frac{\delta}{\delta x^j}+\dot{\theta}^{a}\otimes \frac{\partial}{\partial y^a},
\end{eqnarray}
where
\begin{eqnarray*}
	\theta^j&=&\hat{\Gamma}^j_{lc}\delta y^c\wedge dx^l,\\
	\dot{\theta}^a&=&\frac{1}{2}\left[\frac{\delta}{\delta x^j}(\hat{\Gamma}^a_k)-\frac{\delta}{\delta x^k}(\hat{\Gamma}^a_j)\right]dx^j\wedge dx^k+\left[\frac{\partial}{\partial y^b}(\hat{\Gamma}^a_k)-\hat{\Gamma}^a_{b;k}\right]\delta y^b \wedge dx^k.
\end{eqnarray*}

The curvature operator $\hat{\Omega}\in\mathcal{X}(\wedge^2T_{\mathbb{R}}^\ast \tilde{M}\otimes T_{\mathbb{R}}^\ast\tilde{M}\times T_{\mathbb{R}}\tilde{M})$ of $\nabla$ is given by
$$\hat{\Omega}=\hat{\Omega}_j^k\otimes \left[dx^j\otimes \frac{\delta}{\delta x^k}+\delta y^j\otimes \frac{\partial}{\partial y^k}\right],$$
 where
$$\hat{\Omega}_{j}^{k} = d\hat{\omega}_{j}^{k} - \hat{\omega}_{j}^{l} \wedge \hat{\omega}_{l}^{k}. $$

The horizontal flag curvature of $\nabla$ at $y\in\tilde{M}$ is given by
$$R_y(H,K)=\langle \hat{\Omega} (\hat{\chi}(y),H)K\vert \hat{\chi}(y)\rangle,\quad \forall H,K \in (\mathcal{H}_{\mathbb{R}})_y. $$

\subsection{The Chern-Finsler connection}

The study of complex Finsler geometry hinges on the introduction of a complex horizontal subbundle $\mathcal{H}_{\mathbb{C}}\subset T_{\mathbb{C}}\tilde{M}$ that is $\pmb{J}$-invariant, conjugation invariant, and satisfies the decomposition:
$$T_{\mathbb{C}}\tilde{M}=\mathcal{H}_{\mathbb{C}}\oplus\mathcal{V}_{\mathbb{C}}, \quad\mbox{and dually}\quad T_{\mathbb{C}}^\ast\tilde{M}=\mathcal{H}_{\mathbb{C}}^\ast\oplus\mathcal{V}_{\mathbb{C}}^\ast.$$
 Such a bundle $\mathcal{H}_{\mathbb{C}}$ allows one to define a complex non-linear connection $\tilde{D}_{\mathcal{H}_{\mathbb{C}}}$ on $M$. This connection, in turn, determines a complex horizontal map $\Theta^{\tilde{D}_{\mathcal{H}_{\mathbb{C}}}}$, whose image is precisely the original bundle $\mathcal{H}_{\mathbb{C}}$. This establishes a one-to-one correspondence among complex horizontal bundles, complex non-linear connections and complex horizontal maps \cite{AP}. We  note that a complex horizontal bundle is completely determined by its $(1,0)$-part $\mathcal{H}^{1,0}$, yielding the decomposition $T^{1,0}\tilde{M}=\mathcal{H}^{1,0}\oplus\mathcal{V}^{1,0}$, where $\mathcal{V}^{1,0}=\ker d\pi\subset T^{1,0}\tilde{M}$. In general, $\mathcal{H}^{1,0}$ may not comes from a strongly pseudoconvex complex Finsler metric.

However, for a given strongly pseudoconvex complex Finsler metric $F:T^{1,0}M\rightarrow[0,+\infty)$, there is a canonical way to construct $\mathcal{H}^{1,0}$ from $F$, giving rise to the associated complex horizontal bundle $\mathcal{H}_{\mathbb{C}}=\mathcal{H}^{1,0}\oplus\mathcal{H}^{0,1}$. This bundle is called the complex horizontal bundle associated to the Chern-Finsler connection of $F$. More precisely,
$
\mathcal{H}^{1,0}=\mbox{Span}_{\mathbb{C}}\{\delta_1,\cdots,\delta_n\}
$
where
$\delta_\mu=\frac{\partial}{\partial z^\mu}-\Gamma_{;\mu}^\alpha\frac{\partial}{\partial v^\alpha}$ for $\mu=1,\cdots,n$ and the coefficients
\begin{equation}
\Gamma_{;\mu}^\alpha=G^{\bar{\tau}\alpha}\frac{\partial^2F^2}{\partial z^\mu\partial\bar{v}^\tau}\label{HVB}
\end{equation}
are called the complex non-linear connection coefficients of the Chern-Finsler connection associated to $F$.

 Let $\langle \cdot,\cdot\rangle$ denote the Hermitian inner product induced by $F$ on the holomorphic vertical subbundle $\mathcal{V}^{1,0}\cong\mathcal{V}_{\mathbb{R}}$. For any $V=V^\alpha\frac{\partial}{\partial v^\alpha}$, $W=W^\beta\frac{\partial}{\partial v^\beta}$ in $(\mathcal{V}^{1,0})_{(z;v)}$,
\begin{equation}
\langle V, W\rangle=G_{\alpha\bar{\beta}}(z;v)V^\alpha \overline{W^\beta}.
\end{equation}

\begin{theorem}\cite{AP}\label{CFC} Let $F:T^{1,0}M\rightarrow [0,+\infty)$ be a strongly pseudoconvex complex Finsler metric on a complex manifold $M$, and let $\langle\cdot,\cdot\rangle$ be the induced Hermitian structure on $\mathcal{V}^{1,0}$. Then there is a unique complex vertical connection
$$D:\mathcal{X}(\mathcal{V}^{1,0})\rightarrow \mathcal{X}(T_{\mathbb{C}}^\ast\tilde{M}\otimes \mathcal{V}^{1,0})$$ such that
\begin{enumerate}
\item[(i)] For all $X\in T^{1,0}\tilde{M}$ and $V,W\in\mathcal{X}(\mathcal{V}^{1,0})$,
\begin{equation}
X\langle V,W\rangle=\langle D_XV,W\rangle+\langle V,D_{\overline{X}}W\rangle;\label{MC}
\end{equation}
\item[(ii)] $D$ is good.
\end{enumerate}
\end{theorem}

The connection $D$ in Theorem \ref{CFC} is called the \emph{Chern-Finsler connection} of $F$. Its connection $1$-forms $\omega_\beta^\alpha$  are given by
$$\omega_\beta^\alpha=\Gamma_{\beta;\mu}^\alpha dz^\mu+\Gamma_{\beta\gamma}^\alpha\psi^\gamma,$$
 where $\psi^\gamma=dv^\gamma+\Gamma_{;\mu}^\gamma dz^\mu$ and the coefficients
 \begin{equation}
\Gamma_{\beta;\mu}^\alpha=G^{\bar{\tau}\alpha}\delta_\mu(G_{\beta\bar{\tau}}),\quad \Gamma_{\beta\gamma}^\alpha=G^{\bar{\tau}\alpha}\frac{\partial}{\partial v^\gamma}(G_{\beta\bar{\tau}})=\Gamma_{\gamma\beta}^\alpha.\label{HVC}
\end{equation}
are called  the horizontal and vertical connection coefficients of $D$, respectively.

Several important classes of complex Finsler metrics are defined via the horizontal connection coefficients $\Gamma_{\beta;\mu}^\alpha$:
\begin{itemize}
\item $F$ is called a \emph{K\"ahler-Finsler metric} (cf.  \cite{AP}, \cite{Chen-Shen}) if
\begin{equation}
\Gamma_{\beta;\mu}^\alpha-\Gamma_{\mu;\beta}^\alpha=0,\label{KFD}
\end{equation}
which is equivalent to the condition $(\Gamma_{\beta;\mu}^\alpha-\Gamma_{\mu;\beta}^\alpha)v^\beta=0$ .
\item $F$ is called a \emph{weakly K\"ahler-Finsler metric} \cite{AP} if
\begin{equation}
G_\alpha(\Gamma_{\beta;\mu}^\alpha-\Gamma_{\mu;\beta}^\alpha)v^\beta=0;
\end{equation}
\item $F$ is called a \emph{complex Berwald metric} \cite{Aikou-a} if each $\Gamma_{\beta;\mu}^\alpha$ depends only on the base point $z$;
 \item $F$ is called a \emph{complex locally Minkowski metric} if  each $\Gamma_{\beta;\mu}^\alpha$ depends locally only on the fiber coordinate $v$.
 \end{itemize}

A complex Berwald manifold $(M,F)$ is also characterized by being  \emph{modeled on a complex Minkowski space}  \cite{Aikou-b}.

\begin{definition}
Let $F:T^{1,0}M\rightarrow [0,+\infty)$ be a strongly convex complex Finsler metric on a complex manifold $M$. If $F$ is both a K\"ahler-Finsler metric and a complex Berwald metric, then $F$ is called a \emph{strongly convex K\"ahler-Berwald metric}, and  $(M,F)$ is called a \emph{strongly convex K\"ahler-Berwald manifold}.
\end{definition}

\begin{remark}
While K\"ahler metrics represent the optimal symmetry in Hermitian geometry, our recent work  \cite{Zhong-b, Lin-Zhong, Ge-Zhong,Zhong-c} reveals a similar phenomenon in complex Finsler geometry: the only invariant, strongly pseudoconvex metrics admissible on symmetric complex manifolds are K\"ahler-Berwald metrics. This fact establishes K\"ahler-Berwald metrics as the natural counterpart to K\"ahler metrics in Finsler setting and naturally raises the problem of characterizing them among all complex Finsler metrics.
\end{remark}

Let $\langle\langle\cdot,\cdot\rangle\rangle$ denote the symmetric product induced by $F$ on $\mathcal{V}^{1,0}$, defined for $V=V^\alpha\frac{\partial}{\partial v^\alpha},W=W^\beta\frac{\partial}{\partial v^\beta}\in(\mathcal{V}^{1,0})_{(z;v)}$ by
\begin{equation}
\langle\langle V,W\rangle\rangle=G_{\alpha\beta}(z;v)V^\alpha W^\beta.
\end{equation}

\begin{proposition}\cite{AP}\label{PS}
 Let $F:T^{1,0}M\rightarrow [0,+\infty)$ be a strongly convex complex Finsler metric on a complex manifold $M$. Then
\begin{equation}
\langle V^o\vert W^o\rangle=\mbox{Re}\,[\langle V, W\rangle+\langle\langle V,W\rangle\rangle],\quad\forall V,W\in \mathcal{V}^{1,0}.\label{VW-a}
\end{equation}
\end{proposition}

\begin{remark}
In applications of the above identity \eqref{VW-a}, the presence of the symmetric product $\langle\langle V,W\rangle\rangle$ possibly causes inconvenience. Fortunately, this drawback can be overcome by the following Lemma \ref{L-a}.
\end{remark}

More precisely, extending the real inner product $\langle \cdot\vert \cdot\rangle$ to the complexification
$\mathcal{V}_{\mathbb{C}}=\mathcal{V}_{\mathbb{R}}\otimes \mathbb{C}$ in the following natural way:
\begin{equation}
\langle X_1+iY_1\vert X_2+iY_2\rangle=\langle X_1\vert X_2\rangle-\langle Y_1\vert Y_2\rangle+i(\langle X_1\vert Y_2\rangle+\langle Y_1\vert X_2\rangle).\label{CLE}
\end{equation}
This extension (still denoted $\langle\cdot\vert \cdot\rangle$) is $\mathbb{C}$-bilinear. In particular, it is $\mathbb{C}$-bilinear on $\mathcal{V}^{1,0}$ and $\mathcal{V}^{1,0}$, since $\mathcal{V}_{\mathbb{C}}=\mathcal{V}^{1,0}\oplus\mathcal{V}^{0,1}$.

The following lemma plays a crucial role in the proof of Theorem \ref{Th-c}  and Theorem \ref{Th-cc}.
\begin{lemma}\label{L-a}
 Let $F:T^{1,0}M\rightarrow [0,+\infty)$ be a strongly convex complex Finsler metric on a complex manifold $M$. Then
\begin{equation}\label{eq-rh}
\langle V, W\rangle=2\langle V\vert \overline{W}\rangle,\quad \forall V,W\in\mathcal{V}^{1,0}.
\end{equation}
\end{lemma}

\begin{proof}
Let $V,W\in\mathcal{V}^{1,0}$ be written as
\begin{equation}
V=X_o=\frac{1}{2}(X-i\pmb{J}X),\quad W=Y_o=\frac{1}{2}(Y-i\pmb{J}Y)\label{VW}
\end{equation}
for real vertical vector fields $X, Y\in \mathcal{V}_{\mathbb{R}}$. Identifying $X=V^o$, $Y=W^o$, $\pmb{J}X=(iV)^o$, and $\pmb{J}Y=(iW)^o$, and substituting  into \eqref{VW-a}, we obtain
\begin{eqnarray*}
		\langle X\vert Y\rangle&=&\text{Re}\,[\langle V,W\rangle+\langle\!\langle V,W \rangle\!\rangle],\\
\langle \pmb{J}X\vert \pmb{J}Y\rangle&=&\text{Re}\,[\langle V,W\rangle-\langle\!\langle V,W \rangle\!\rangle],\\
\langle X\vert \pmb{J}Y\rangle&=&\text{Im}\,[\langle V,W\rangle-\langle\!\langle V,W \rangle\!\rangle],\\
\langle \pmb{J}X\vert Y\rangle&=&-\text{Im}\,[\langle V,W\rangle+\langle\!\langle V,W \rangle\!\rangle].
\end{eqnarray*}
From these, it follows that
\begin{eqnarray*}
\langle X\vert Y\rangle+\langle \pmb{J}X\vert \pmb{J}Y\rangle&=&2\text{Re}[\langle V,W\rangle],\\
 \langle X\vert \pmb{J}Y\rangle-\langle \pmb{J}X\vert Y\rangle&=&2\text{Im}[\langle V,W\rangle].
\end{eqnarray*}
Therefore,
$$\langle X\vert Y\rangle+\langle \pmb{J}X\vert JY\rangle+i\langle X\vert \pmb{J}Y\rangle-i\langle \pmb{J}X\vert Y\rangle=2\langle V,W\rangle.
$$
On the other hand, using the extension  \eqref{CLE}, we compute
$$
\langle V\vert \overline{W}\rangle=\frac{1}{4}\left\{\langle X\vert Y\rangle+\langle \pmb{X}\vert \pmb{J}Y\rangle+i\langle X\vert \pmb{J}Y\rangle-i\langle \pmb{J}X\vert Y\rangle\right\}.
$$
Comparing the two expressions yields the desired identity \eqref{eq-rh}.

\end{proof}

\begin{remark}
Via the horizontal map $\hat{\Theta}:\mathcal{V}_{\mathbb{R}}\rightarrow\mathcal{H}_{\mathbb{R}}$, the inner product $\langle\cdot\vert \cdot\rangle$ can also be extended to the complexification $\hat{\mathcal{H}}_{\mathbb{C}}:=\mathcal{H}_{\mathbb{R}}\otimes \mathbb{C}$. In general, however, $\hat{\mathcal{H}}_{\mathbb{C}}\neq \mathcal{H}^{1,0}\oplus \mathcal{H}^{0,1}$.
\end{remark}

\section{Parallelism of $\pmb{J}$ with respect to the Cartan connection}

In this section, we mainly investigate the properties of strongly convex complex Finsler manifolds $(M,F)$ under the condition that the canonical complex structure $\pmb{J}$ on the complex manifold $T^{1,0}M$ satisfies some different extent of parallelism with respect to the Cartan connection $\nabla$ associated to $F$.

 Firstly, we obtain the following  necessary and sufficient condition for the  parallelism of $\pmb{J}$ with respect to $\nabla$.

\begin{theorem}\label{Th-ab}
Let $F: T^{1,0}M\rightarrow [0,+\infty)$ be a strongly convex complex Finsler metric on a complex manifold $M$. Then $\nabla \pmb{J}=0$ iff
\begin{eqnarray}
\hat{\Gamma}_{\alpha^\ast   ;j}^\beta&=&-\hat{\Gamma}_{\alpha; j}^{\beta^\ast  },\quad\hat{\Gamma}_{\alpha^\ast   ;j}^{\beta^\ast  }=\Gamma_{\alpha;j}^\beta,\label{eq-ab}\\
\hat{\Gamma}_{\alpha^\ast   j}^\beta&=&-\hat{\Gamma}_{\alpha j}^{\beta^\ast  }, \quad \hat{\Gamma}_{\alpha^\ast   j}^{\beta^\ast  }=\hat{\Gamma}_{\alpha j}^\beta,\label{eq-abc}
\end{eqnarray}
for all $\alpha,\beta=1,\cdots,n$ and $j=1,\cdots,2n$.
\end{theorem}

\begin{proof}
Since
$$(\nabla_X \pmb{J} )(Y)=\nabla_X( \pmb{J} Y)- \pmb{J} (\nabla_XY)$$
for all $X\in T_{\mathbb{R}}\tilde{M}$ and $Y\in\mathcal{X}(\mathcal{V}_{\mathbb{R}})$,
the condition that $\nabla\pmb{J} =0$ is  equivalent to
	\begin{equation}
		\nabla_X( \pmb{J} Y)= \pmb{J} (\nabla_XY)\label{pc}
	\end{equation}
	for all $X\in T_{\mathbb{R}}\tilde{M}$  and $Y\in\mathcal{X}(\mathcal{V}_{\mathbb{R}})$.

Taking $X=\frac{\delta}{\delta x^j}, Y=\frac{\partial}{\partial y^k}$ and substituting them into \eqref{pc} and using \eqref{CS-T}, we obtain
	\begin{eqnarray}
		\nabla_{\frac{\delta}{\delta x^j}}\left( \pmb{J} \frac{\partial}{\partial y^k}\right)&=&\left\{
		\begin{array}{ll}
			\hat{\Gamma}_{\alpha^\ast   ;j}^\beta\frac{\partial}{\partial y^\beta}+\hat{\Gamma}_{\alpha^\ast   ;j}^{\beta^\ast  }\frac{\partial}{\partial y^{\beta^\ast  }}, & k=\alpha, \\
			-\hat{\Gamma}_{\alpha; j}^\beta\frac{\partial}{\partial y^\beta}-\hat{\Gamma}_{\alpha ;j}^{\beta^\ast  }\frac{\partial}{\partial y^{\beta^\ast  }}, & k=\alpha^\ast
		\end{array}
		\right.
		\label{eq-ha}
	\end{eqnarray}
	and
	\begin{eqnarray}
		\pmb{J}\left(\nabla_{\frac{\delta}{\delta x^j}}\frac{\partial}{\partial y^k}\right)
		=\left\{
		\begin{array}{ll}
			\hat{\Gamma}_{\alpha; j}^\beta\frac{\partial}{\partial y^{\beta^\ast  }}-\hat{\Gamma}_{\alpha; j}^{\beta^\ast  }\frac{\partial}{\partial y^\beta}, & k=\alpha, \\
			\hat{\Gamma}_{\alpha^\ast   ; j}^\beta\frac{\partial}{\partial y^{\beta^\ast  }}-\hat{\Gamma}_{\alpha^\ast   ; j}^{\beta^\ast  }\frac{\partial}{\partial y^\beta}, & k=\alpha^\ast   .
		\end{array}
		\right.
		\label{eq-hb}
	\end{eqnarray}
Comparing \eqref{eq-ha} and \eqref{eq-hb}, we see  that \eqref{pc} holds for all $X\in\mathcal{H}_{\mathbb{R}}$  and $Y\in\mathcal{X}(\mathcal{V}_{\mathbb{R}})$ iff the equalities in
\eqref{eq-ab} hold.
	
Taking $X=\frac{\partial}{\partial y^j}, Y=\frac{\partial}{\partial y^k}$ and substituting them into \eqref{pc} and using \eqref{CS-T}, we obtain	
 \begin{eqnarray}
	\nabla_{\frac{\partial }{\partial y^j}}\left( \pmb{J} \frac{\partial}{\partial y^k}\right)&=&\left\{
	\begin{array}{ll}
		\hat{\Gamma}_{\alpha^\ast   j}^\beta\frac{\partial}{\partial y^\beta}+\hat{\Gamma}_{\alpha^\ast   j}^{\beta^\ast  }\frac{\partial}{\partial y^{\beta^\ast  }}, & k=\alpha, \\
		-\hat{\Gamma}_{\alpha j}^\beta\frac{\partial}{\partial y^\beta}-\hat{\Gamma}_{\alpha j}^{\beta^\ast  }\frac{\partial}{\partial y^{\beta^\ast  }}, & k=\alpha^\ast
	\end{array}
	\right.
\label{eq-hc}
\end{eqnarray}
and
\begin{eqnarray}
	\pmb{J}\left(\nabla_{\frac{\partial}{\partial y^j}}\frac{\partial}{\partial y^k}\right)
	=\left\{
	\begin{array}{ll}
		\hat{\Gamma}_{\alpha j}^\beta\frac{\partial}{\partial y^{\beta^\ast  }}-\hat{\Gamma}_{\alpha j}^{\beta^\ast  }\frac{\partial}{\partial y^\beta}, & k=\alpha, \\
	\hat{\Gamma}_{\alpha^\ast    j}^\beta\frac{\partial}{\partial y^{\beta^\ast  }}-\hat{\Gamma}_{\alpha^\ast    j}^{\beta^\ast  }\frac{\partial}{\partial y^\beta}, & k=\alpha^\ast   .
	\end{array}
	\right.
\label{eq-hd}
\end{eqnarray}
Comparing \eqref{eq-hc} and \eqref{eq-hd}, we see  that \eqref{pc} holds for all $X\in\mathcal{V}_{\mathbb{R}}$  and $Y\in\mathcal{X}(\mathcal{V}_{\mathbb{R}})$ iff
\eqref{eq-abc} hold.
\end{proof}

\begin{corollary}\label{cor-a}
Let $F: T^{1,0}M\rightarrow [0,+\infty)$ be a strongly convex complex Finsler metric on a complex manifold $M$.
If $\pmb{J}$ is horizontal parallel with respect to $\nabla$, namely $\nabla_X \pmb{J}=0$ for all $X\in\mathcal{H}_{\mathbb{R}}$, then
\begin{eqnarray}
\hat{\Gamma}_{\alpha^\ast   }^\beta=-\hat{\Gamma}_{\alpha}^{\beta^\ast  },\quad \hat{\Gamma}_{\alpha^\ast   }^{\beta^\ast  }=\hat{\Gamma}_{\alpha}^\beta \label{eq-cc}
\end{eqnarray}
for $\alpha=1,\cdots,n$.
\end{corollary}
\begin{proof}
 By  Theorem  \ref{Th-ab},  $\nabla_X \pmb{J}=0$ for all $X\in\mathcal{H}_{\mathbb{R}}$ iff the equalities in \eqref{eq-ab} hold. Since the Cartann connection coefficients are symmetric with respect to  lower indices, namely $\hat{\Gamma}_{j;l}^k=\hat{\Gamma}_{l;j}^k$, and they satisfy $\hat{\Gamma}_{j;l}^ky^j=\hat{\Gamma}_{l}^k$. Contracting the equalities in \eqref{eq-ab} with $y^j$, we immediately obtain \eqref{eq-cc}.
\end{proof}

\begin{definition}Let $\pmb{J}$ be the canonical structure on the complex manifold $T^{1,0}M$. The real horizontal bundle $\mathcal{H}_{\mathbb{R}}$ is called $ \pmb{J}$-invariant if
$$ \pmb{J}\left(\frac{\delta}{\delta x^\alpha}\right)=\frac{\delta}{\delta x^{\alpha^\ast   }},\quad \pmb{J}\left(\frac{\delta}{\delta x^{\alpha^\ast   }}\right)=-\frac{\delta}{\delta x^\alpha},\quad\alpha=1,\cdots,n.
$$
\end{definition}

\begin{proposition}\cite{AP}
$\mathcal{H}_{\mathbb{R}}$ is $ \pmb{J}$-invariant iff the equalities in \eqref{eq-cc} hold.
\end{proposition}
\begin{proof}
 Indeed,
 \begin{eqnarray*}
\pmb{J}\left(\frac{\delta}{\delta x^\alpha}\right)&=&\frac{\partial}{\partial x^{\alpha^\ast   }}-\hat{\Gamma}_{\alpha}^\beta\frac{\partial}{\partial x^{\beta^\ast  }}+\hat{\Gamma}_{\alpha}^{\beta^\ast  }\frac{\partial}{\partial x^\beta},\quad\frac{\delta}{\delta x^{\alpha^\ast   }}
=\frac{\partial}{\partial x^{\alpha^\ast   }}-\hat{\Gamma}_{\alpha^\ast   }^{\beta^\ast  }\frac{\partial}{\partial x^{\beta^\ast  }}-\hat{\Gamma}_{\alpha^\ast   }^\beta\frac{\partial}{\partial x^\beta},\\
  \pmb{J}\left(\frac{\delta}{\delta x^{\alpha^\ast   }}\right)&=&-\frac{\partial}{\partial x^{\alpha}}-\hat{\Gamma}_{\alpha^\ast   }^\beta\frac{\partial}{\partial x^{\beta^\ast  }}+\hat{\Gamma}_{\alpha^\ast   }^{\beta^\ast  }\frac{\partial}{\partial x^\beta},\quad-\frac{\delta}{\delta x^\alpha}  =-\frac{\partial}{\partial x^\alpha}+\hat{\Gamma}_{\alpha}^{\beta^\ast  }\frac{\partial}{\partial x^{\beta^\ast  }}+\hat{\Gamma}_{\alpha}^{\beta}\frac{\partial}{\partial x^\beta}.
 \end{eqnarray*}
Thus $\mathcal{H}_{\mathbb{R}}$ is $ \pmb{J}$-invariant iff the equalities in \eqref{eq-cc} hold.
\end{proof}

 Let $\widetilde{\mathcal{H}}_{\mathbb{C}}=\mathcal{H}_{\mathbb{R}}\otimes\mathbb{C}$ be the complexified horizontal bundle of $\mathcal{H}_{\mathbb{R}}$. Formally, we have
 \begin{eqnarray*}
\left(\frac{\delta}{\delta x^\alpha}\right)_o=\frac{1}{2}\left(\frac{\delta}{\delta x^\alpha}-i \pmb{J}\frac{\delta}{\delta x^\alpha}\right),\,
\left(\frac{\delta}{\delta x^{\alpha^\ast   }}\right)_o=\frac{1}{2}\left(\frac{\delta}{\delta x^{\alpha^\ast   }}-i \pmb{J}\frac{\delta}{\delta x^{\alpha^\ast }}\right), \,\alpha=1,\cdots,n.
\end{eqnarray*}
In general,  $\left(\frac{\delta}{\delta x^\alpha}\right)_o\neq \delta_\alpha$ and $\left(\frac{\delta}{\delta x^{\alpha^\ast   }}\right)_o\neq i\delta_\alpha$, as pointed out on page 114 in \cite{AP}. However, if  $\mathcal{H}_{\mathbb{R}}$ is $ \pmb{J}$-invariant,  then we are able to obtain a complex horizontal bundle, denoted by $\widetilde{\mathcal{H}}^{1,0}$.

For this purpose,  we define
\begin{equation}
\mathcal{N}_{;\alpha}^\beta:=\hat{\Gamma}_{\alpha}^\beta+i\hat{\Gamma}_{\alpha}^{\beta^\ast  }
\end{equation}
and
\begin{equation}
\frac{\delta}{\delta z^\alpha}:=\frac{\partial}{\partial z^\alpha}-\mathcal{N}_{;\alpha}^\beta\frac{\partial}{\partial v^\beta},\quad \alpha=1,\cdots,n.
\end{equation}

\begin{proposition}\label{pro-b}
If $\mathcal{H}_{\mathbb{R}}$ is $ \pmb{J}$-invariant, then
\begin{equation}
\left(\frac{\delta}{\delta x^\alpha}\right)_o=\frac{\delta}{\delta z^\alpha}\quad\text{and}\quad \left(\frac{\delta}{\delta x^{\alpha^\ast   }}\right)_o=i\frac{\delta}{\delta z^\alpha},
\end{equation}
or equivalently
\begin{equation}
\left(\frac{\delta}{\delta z^\alpha}\right)^o=\frac{\delta}{\delta x^\alpha}\quad\text{and}\quad \left(i\frac{\delta}{\delta z^\alpha}\right)^o=\frac{\delta}{\delta x^{\alpha^\ast   }}.
\end{equation}
\end{proposition}
\begin{proof}
Direct computation shows
\begin{eqnarray*}
\left(\frac{\delta}{\delta x^\alpha}\right)_o&=&\frac{1}{2}\left(\frac{\delta}{\delta x^\alpha}-i \pmb{J}\frac{\delta}{\delta x^\alpha}\right)\\
&=&\frac{1}{2}\left[\frac{\partial}{\partial x^\alpha}-\hat{\Gamma}_{\alpha}^\beta\frac{\partial}{\partial y^\beta}-\hat{\Gamma}_{\alpha}^{\beta^\ast  }\frac{\partial}{\partial y^{\beta^\ast  }}-i\left(\frac{\partial}{\partial x^{\alpha^\ast   }}-\hat{\Gamma}_{\alpha}^\beta\frac{\partial}{\partial y^{\beta^\ast  }}+\hat{\Gamma}_{\alpha}^{\beta^\ast  }\frac{\partial}{\partial y^\beta}\right)\right]\\
&=&\frac{1}{2}\left[\left(\frac{\partial}{\partial x^\alpha}-i\frac{\partial}{\partial x^{\alpha^\ast   }}\right)-\left(\hat{\Gamma}_{\alpha}^\beta+i\hat{\Gamma}_{\alpha}^{\beta^\ast  }\right)\frac{\partial}{\partial y^\beta}+i\left(\hat{\Gamma}_{\alpha}^\beta+i\hat{\Gamma}_{\alpha}^{\beta^\ast  }\right)\frac{\partial}{\partial y^{\beta^\ast  }}\right]\\
&=&\frac{1}{2}\left[\left(\frac{\partial}{\partial x^\alpha}-i\frac{\partial}{\partial x^{\alpha^\ast   }}\right)-\left(\hat{\Gamma}_{\alpha}^\beta+i\hat{\Gamma}_{\alpha}^{\beta^\ast  }\right)\left(\frac{\partial}{\partial y^\beta}-i\frac{\partial}{\partial y^{\beta^\ast  }}\right)\right]\\
&=&\frac{\partial}{\partial z^\alpha}-\left(\hat{\Gamma}_{\alpha}^\beta+i\hat{\Gamma}_{\alpha}^{\beta^\ast  }\right)\frac{\partial}{\partial v^\beta}\\
&=&\frac{\delta}{\delta z^\alpha}
\end{eqnarray*}
and
\begin{eqnarray*}
\left(\frac{\delta}{\delta x^{\alpha^\ast   }}\right)_o&=&\frac{1}{2}\left(\frac{\delta}{\delta x^{\alpha^\ast   }}-i \pmb{J}\frac{\delta}{\delta x^{\alpha^\ast   }}\right)\\
&=&\frac{1}{2}\left[\frac{\partial}{\partial x^{\alpha^\ast   }}-\hat{\Gamma}_{\alpha^\ast   }^\beta\frac{\partial}{\partial y^\beta}-\hat{\Gamma}_{\alpha^\ast   }^{\beta^\ast  }\frac{\partial}{\partial y^{\beta^\ast  }}-i\left(-\frac{\partial}{\partial x^\alpha}-\hat{\Gamma}_{\;\alpha^\ast   }^\beta\frac{\partial}{\partial y^{\beta^\ast  }}+\hat{\Gamma}_{\alpha^\ast   }^{\beta^\ast  }\frac{\partial}{\partial y^{\beta}}\right)\right]\\
&=&\frac{1}{2}\left[i\left(\frac{\partial}{\partial x^\alpha}-i\frac{\partial}{\partial x^{\alpha^\ast   }}\right)-\left(\hat{\Gamma}_{\alpha^\ast   }^\beta+i\hat{\Gamma}_{\alpha^\ast   }^{\beta^\ast  }\right)\frac{\partial}{\partial y^\beta}-\left(\hat{\Gamma}_{\alpha^\ast   }^{\beta^\ast  }-i\hat{\Gamma}_{\alpha^\ast   }^\beta\right)\frac{\partial}{\partial y^{\beta^\ast  }}\right]\\
&=&\frac{1}{2}\left[i\left(\frac{\partial}{\partial x^\alpha}-i\frac{\partial}{\partial x^{\alpha^\ast   }}\right)-i\left(\hat{\Gamma}_{\alpha^\ast   }^{\beta^\ast  }-i\hat{\Gamma}_{\alpha^\ast   }^\beta\right)\frac{\partial}{\partial y^\beta}-\left(\hat{\Gamma}_{\alpha^\ast   }^{\beta^\ast  }-i\hat{\Gamma}_{\;\alpha^\ast   }^\beta\right)\frac{\partial}{\partial y^{\beta^\ast  }}\right]\\
&=&i\frac{1}{2}\left[\left(\frac{\partial}{\partial x^\alpha}-i\frac{\partial}{\partial x^{\alpha^\ast   }}\right)-\left(\hat{\Gamma}_{\alpha^\ast   }^{\beta^\ast  }-i\hat{\Gamma}_{\alpha^\ast   }^\beta\right)\left(\frac{\partial}{\partial y^\beta}-i\frac{\partial}{\partial y^{\beta^\ast  }}\right)\right]\\
&=&i\left[\frac{\partial}{\partial z^\alpha}-\left(\hat{\Gamma}_{\alpha^\ast   }^{\beta^\ast  }-i\hat{\Gamma}_{\alpha^\ast   }^\beta\right)\frac{\partial}{\partial v^\beta}\right]\\
&=&i\left[\frac{\partial}{\partial z^\alpha}-\left(\hat{\Gamma}_{\alpha}^{\beta}+i\hat{\Gamma}_{\alpha}^{\beta^\ast  }\right)\frac{\partial}{\partial v^\beta}\right]\\
&=&i\frac{\delta}{\delta z^\alpha}.
\end{eqnarray*}
\end{proof}

Let $\hat{\Gamma}_{j;l}^k$ and $\hat{\Gamma}_{jl}^k$ be the horizontal and vertical Cartan connection coefficients of $F$. Define
\begin{equation}
\mathcal{N}_{\alpha;\gamma}^\beta:=\hat{\Gamma}_{\alpha;\gamma}^\beta+i\hat{\Gamma}_{\alpha;\gamma}^{\beta^\ast  } \quad \text{and }\quad \mathcal{N}_{\alpha\gamma}^\beta:=\hat{\Gamma}_{\alpha\gamma}^\beta+i\hat{\Gamma}_{\alpha\gamma}^{\beta^\ast  },\quad \alpha,\beta,\gamma=1,\cdots,n.
\end{equation}

\begin{remark}\label{R-a}
Since $\hat{\Gamma}_{j;l}^k=\hat{\Gamma}_{l;j}^k$ and $\hat{\Gamma}_{jl}^k=\hat{\Gamma}_{lj}^k$, it is clear that
\begin{equation}
\mathcal{N}_{\alpha;\gamma}^\beta=\mathcal{N}_{\gamma;\alpha}^\beta \quad \mbox{and}\quad \mathcal{N}_{\alpha\gamma}^\beta=\mathcal{N}_{\gamma\alpha}^\beta.\label{SN}
\end{equation}
\end{remark}
\begin{proposition}\label{prop3.3}
If $\nabla_X\pmb{J}=0$ for all $X\in\mathcal{H}_{\mathbb{R}}$, then
\begin{equation}
\mathcal{N}_{\alpha;\gamma}^\beta v^\alpha=\mathcal{N}_{;\gamma}^\beta;
\end{equation}
if $\nabla_Y\pmb{J}=0$ for all $Y\in\mathcal{V}_{\mathbb{R}}$, then
\begin{equation}
\mathcal{N}_{\alpha\gamma}^\beta v^\alpha=0.
\end{equation}
\end{proposition}

\begin{proof}
Indeed, since $v^\alpha=y^\alpha+iy^{\alpha^\ast}$,
\begin{eqnarray*}
	\mathcal{N}_{\alpha;\gamma}^\beta v^\alpha
%&=&(\hat{\Gamma}_{\alpha;\gamma}^\beta+i\hat{\Gamma}_{\alpha;\gamma}^{\beta^\ast  })(y^\alpha+iy^{\alpha^\ast   })\nonumber\\
	&=&(\hat{\Gamma}_{\alpha;\gamma}^\beta y^\alpha-\hat{\Gamma}_{\alpha;\gamma}^{\beta^\ast  }y^{\alpha^\ast   })+i(\hat{\Gamma}_{\alpha;\gamma}^\beta y^{\alpha^\ast   }+\hat{\Gamma}_{\alpha;\gamma}^{\beta^\ast  }y^\alpha).
\end{eqnarray*}
By assumption $\nabla_X\pmb{J}=0$ for all $X\in\mathcal{H}_{\mathbb{R}}$, thus we have \eqref{eq-ab}. So that
\begin{eqnarray*}
	\mathcal{N}_{\alpha;\gamma}^\beta v^\alpha&=&(\hat{\Gamma}_{\alpha;\gamma}^\beta y^\alpha+\hat{\Gamma}_{\alpha^\ast   ;\gamma}^{\beta}y^{\alpha^\ast   })+i(\hat{\Gamma}_{\alpha^\ast   ;\gamma}^{\beta^\ast  } y^{\alpha^\ast   }+\hat{\Gamma}_{\alpha;\gamma}^{\beta^\ast  }y^\alpha)
=\hat{\Gamma}_\gamma^\beta+i\hat{\Gamma}_\gamma^{\beta^\ast}=\mathcal{N}_{;\gamma}^\beta,
\end{eqnarray*}
since the horizontal Cartan connection coefficients satisfy $\hat{\Gamma}_{j;l}^ky^j=\hat{\Gamma}_{l}^k$.

Similarly,
\begin{eqnarray*}
	\mathcal{N}_{\alpha\gamma}^\beta v^\alpha
%&=&(\hat{\Gamma}_{\alpha;\gamma}^\beta+i\hat{\Gamma}_{\alpha;\gamma}^{\beta^\ast  })(y^\alpha+iy^{\alpha^\ast   })\nonumber\\
	&=&(\hat{\Gamma}_{\alpha\gamma}^\beta y^\alpha-\hat{\Gamma}_{\alpha\gamma}^{\beta^\ast  }y^{\alpha^\ast   })+i(\hat{\Gamma}_{\alpha\gamma}^\beta y^{\alpha^\ast   }+\hat{\Gamma}_{\alpha\gamma}^{\beta^\ast  }y^\alpha).
\end{eqnarray*}
By assumption $\nabla_Y\pmb{J}=0$ for all $Y\in\mathcal{V}_{\mathbb{R}}$,  thus we  have \eqref{eq-abc}. So that
\begin{eqnarray*}
	\mathcal{N}_{\alpha\gamma}^\beta v^\alpha&=&(\hat{\Gamma}_{\alpha\gamma}^\beta y^\alpha+\hat{\Gamma}_{\alpha^\ast \gamma}^{\beta}y^{\alpha^\ast   })+i(\hat{\Gamma}_{\alpha^\ast \gamma}^{\beta^\ast  } y^{\alpha^\ast   }+\hat{\Gamma}_{\alpha\gamma}^{\beta^\ast  }y^\alpha)=\hat{\Gamma}_{j\gamma}^\beta y^j+i\hat{\Gamma}_{j\gamma}^{\beta^\ast}y^j=0,
\end{eqnarray*}
since the vertical Cartan connection coefficients satisfy $\hat{\Gamma}_{jl}^ky^j=0$.

\end{proof}
\begin{proposition} \label{P-3.4} Let $F: T^{1,0}M\rightarrow [0,+\infty)$ be a strongly convex complex Finsler metric on a complex manifold $M$. Let $U_A$ and $U_B$ be local holomorphic coordinate neighborhoods on $M$ with $U_A\cap U_B\neq \emptyset$, such that $(z_A;v_A)=(z_A^1,\cdots, z_A^n; v_A^1,\cdots, v_A^n)$ and $(z_B;v_B)=(z_B^1,\cdots, z_B^n;v_B^1$, $\cdots$, $v_B^n)$ are the induced holomorphic coordinates on $\pi^{-1}(U_A)$ and $\pi^{-1}(U_B)\subset T^{1,0}M$, respectively. If $\nabla \pmb{J}=0$, then
\begin{eqnarray}
\left(\mathcal{N}_{\alpha;\gamma}^\beta\right)_B&=&\frac{\partial z_A^\mu}{\partial z_B^\alpha}\frac{\partial z_A^\nu}{\partial z_B^\gamma}\left(\mathcal{N}_{\mu;\nu}^\delta\right)_A\frac{\partial z_B^\beta}{\partial z_A^\delta}
-\frac{\partial z_A^\mu}{\partial z_B^\alpha}\frac{\partial z_A^\nu}{\partial z_B^\gamma}\frac{\partial^2 z_B^\beta}{\partial z_A^\mu\partial z_A^\nu},\label{cbc-a}\\
\left(\mathcal{N}_{;\gamma}^\beta\right)_B&=&\frac{\partial z_A^\nu}{\partial z_B^\gamma}\left(\mathcal{N}_{;\nu}^\delta\right)_A\frac{\partial z_B^\beta}{\partial z_A^\delta}
-\frac{\partial z_A^\nu}{\partial z_B^\gamma}\frac{\partial^2 z_B^\beta}{\partial z_A^\mu\partial z_A^\nu}v_A^\mu,\label{cbc-b}\\
	\left(\mathcal{N}_{\alpha\gamma}^\beta\right)_B&=&\frac{\partial z_A^\mu}{\partial z_B^\alpha}\frac{\partial z_A^\nu}{\partial z_B^\gamma}\left(\mathcal{N}_{\mu\nu}^\delta\right)_A\frac{\partial z_B^\beta}{\partial z_A^\delta}.\label{cbc-c}
\end{eqnarray}
\end{proposition}

\begin{proof} Let
\begin{eqnarray*}
z_A^\alpha=x_A^\alpha+ix_A^{\alpha^\ast   },\quad z_B^\alpha=x_B^\alpha+ix_B^{\alpha^\ast   },\quad
v_A^\alpha=y_A^\alpha+iy_A^{\alpha^\ast   },\quad v_B^\alpha=y_B^\alpha+iy_B^{\alpha^\ast   }
\end{eqnarray*}
for $\alpha=1,\cdots,n$. Then $x_A=(x_A^1,\cdots, x_A^{2n})$ and $x_B=(x_B^1,\cdots, x_B^{2n})$ are local real coordinates on $U_A$ and $U_B$, respectively, and $(x_A,y_A)=(x_A^1,\cdots, x_A^{2n};y_A^1,\cdots, y_A^{2n}),(x_B;y_B)=(x_B^1,\cdots,x_B^{2n};y_B^1,\cdots,y_B^{2n})$ are the induced real coordinates on $\pi^{-1}(U_A)$ and $\pi^{-1}(U_B)$, respectively.
Under a change of coordinates on $\pi^{-1}(U_A)\cap \pi^{-1}(U_B)\neq \emptyset$, the horizontal connection coefficients transform as
\begin{equation}\label{eq-clc1}
\left(\hat{\Gamma}_{a;c}^b\right)_B=\sum_{h,k=1}^{2n}\frac{\partial x_A^h}{\partial x_B^a}\frac{\partial x_A^k}{\partial x_B^c}\left[\sum_{j=1}^{2n}\left(\hat{\Gamma}_{h;k}^j\right)_A\frac{\partial x_B^b}{\partial x_A^j}-\frac{\partial^2x_B^b}{\partial x_A^h\partial x_A^k}\right].
\end{equation}
By definition, $\mathcal{N}_{\alpha;\gamma}^\beta=\hat{\Gamma}_{\alpha;\gamma}^\beta+i\hat{\Gamma}_{\alpha;\gamma}^{\beta^\ast  }$, which together with \eqref{eq-clc1} yield
\begin{equation}
\left(\mathcal{N}_{\alpha;\gamma}^\beta\right)_B
=\sum_{h,k=1}^{2n}\frac{\partial x_A^h}{\partial x_B^\alpha}\frac{\partial x_A^k}{\partial x_B^\gamma}\left[\sum_{j=1}^{2n}\left(\hat{\Gamma}_{h;k}^j\right)_A\frac{\partial z_B^\beta}{\partial x_A^j}-\frac{\partial^2z_B^\beta}{\partial x_A^h\partial x_A^k}\right].\label{eq-b}
\end{equation}

Since  $z_B^\beta$ is holomorphic in $z_A^\delta=x_A^\delta+ix_A^{\delta^\ast }$, we have
\begin{equation}\label{eq-zhc}
\frac{\partial z_B^\beta}{\partial x_A^\delta}=\frac{\partial z_B^\beta}{\partial z_A^\delta}\quad \text{and } \quad \frac{\partial z_B^\beta}{\partial x_A^{\delta^\ast}}=i\frac{\partial z_B^\beta}{\partial z_A^\delta}.
\end{equation}
Thus
\begin{eqnarray*}
\sum_{j=1}^{2n}\left(\hat{\Gamma}_{h;k}^j\right)_A\frac{\partial z_B^\beta}{\partial x_A^j}
&=&\sum_{\delta=1}^{n}\left[\left(\hat{\Gamma}_{h;k}^\delta\right)_A\frac{\partial z_B^\beta}{\partial x_A^\delta}
+\left(\hat{\Gamma}_{h;k}^{\delta^\ast }\right)_A\frac{\partial z_B^\beta}{\partial x_A^{\delta^\ast }}\right]\\
&=&\sum_{\delta=1}^n\left[\left(\hat{\Gamma}_{h;k}^\delta\right)_A+i\left(\hat{\Gamma}_{h;k}^{\delta^\ast }\right)_A\right]\frac{\partial z_B^\beta}{\partial z_A^\delta}.
\end{eqnarray*}
%Next,
%\begin{eqnarray*}
%&&\sum_{h,k=1}^{2n}\frac{\partial x_A^h}{\partial x_B^\alpha}\frac{\partial x_A^k}{\partial x_B^\gamma}\left[\sum_{\delta=1}^n\left(\hat{\Gamma}_{h;k}^\delta\right)_A+i\left(\hat{\Gamma}_{h;k}^{\delta^\ast }\right)_A\right]\\
%&=&\sum_{\mu,\nu=1}^n\Bigg\{\frac{\partial x_A^\mu}{\partial x_B^\alpha}\frac{\partial x_A^\nu}{\partial x_B^\gamma}\left[\sum_{\delta=1}^n\left(\hat{\Gamma}_{\mu;\nu}^\delta\right)_A+i\left(\hat{\Gamma}_{\mu;\nu}^{\delta^\ast }\right)_A\right]\\\
%&&+\frac{\partial x_A^\mu}{\partial x_B^\alpha}\frac{\partial x_A^{\nu^\ast }}{\partial x_B^\gamma}\left[\sum_{\delta=1}^n\left(\hat{\Gamma}_{\mu;\nu^\ast }^\delta\right)_A+i\left(\hat{\Gamma}_{\mu;\nu^\ast }^{\d%&&+\frac{\partial x_A^{\mu^\ast  }}{\partial x_B^\alpha}\frac{\partial x_A^\nu}{\partial x_B^\gamma}\left[\sum_{\delta=1}^n\left(\Gamma_{\mu^\ast  ;\nu}^\delta\right)_A+i\left(\hat{\Gamma}_{\mu^\ast  ;\nu}^{\delta^\ast }\right)_A\right]\\
%&&+\frac{\partial x_A^{n+\mu}}{\partial x_B^\alpha}\frac{\partial x_A^{n+\nu}}{\partial x_B^\gamma}\left[\sum_{\delta=1}^n\left(\hat{\Gamma}_{n+\mu; n+\nu}^\delta\right)_A+i\left(\hat{\Gamma}_{n+\mu;\nu^\ast }^{\delta^\ast }\right)_A\right]\Bigg\}.
%\end{eqnarray*}
Using \eqref{eq-ha},  a direct computation shows that
\begin{eqnarray*}
&&\sum_{h,k=1}^{2n}\frac{\partial x_A^h}{\partial x_B^\alpha}\frac{\partial x_A^k}{\partial x_B^\gamma}\sum_{\delta=1}^n\left[\left(\hat{\Gamma}_{h;k}^\delta\right)_A+i\left(\hat{\Gamma}_{h;k}^{\delta^\ast }\right)_A\right]\\
&=&\sum_{\mu,\nu=1}^n\frac{\partial z_A^\mu}{\partial x_B^\alpha}\frac{\partial z_A^\nu}{\partial x_B^\gamma}\sum_{\delta=1}^n\left[\left(\hat{\Gamma}_{\mu;\nu}^\delta\right)_A+i\left(\hat{\Gamma}_{\mu;\nu}^{\delta^\ast }\right)_A\right]\\
&=&\frac{\partial z_A^\mu}{\partial x_B^\alpha}\frac{\partial z_A^\nu}{\partial x_B^\gamma}\left(\mathcal{N}_{\mu;\nu}^\delta\right)_A.
\end{eqnarray*}
Similarly,
$$
\sum_{h,k=1}^{2n}\frac{\partial x_A^h}{\partial x_B^\alpha}\frac{\partial x_A^k}{\partial x_B^\gamma}\frac{\partial^2z_B^\beta}{\partial x_A^h\partial x_A^k}=\sum_{\mu,\nu=1}^n\frac{\partial z_A^\mu}{\partial x_B^\alpha}\frac{\partial z_A^\nu}{\partial z_B^\gamma}\frac{\partial^2z_B^\beta}{\partial z_A^\mu\partial z_A^\nu}.
$$
This proves \eqref{cbc-a}. Contracting \eqref{cbc-a} with $v_B^\alpha=v_A^\gamma\frac{\partial z_B^\alpha}{\partial z_A^\gamma}$ yields \eqref{cbc-b}.

The vertical connection coefficients $\hat{\Gamma}_{jk}^i$ of the Cartan connection transform as follows \cite{AP}
	$$
	\left(\hat{\Gamma}_{ac}^b\right)_B=\sum_{h,k=1}^{2n}\frac{\partial x_A^h}{\partial x_B^a}\frac{\partial x_A^k}{\partial x_B^c}\left[\sum_{j=1}^{2n}\left(\hat{\Gamma}_{hk}^j\right)_A\frac{\partial x_B^b}{\partial x_A^j}\right].
	$$
Using the definition of $\mathcal{N}_{\alpha\gamma}^\beta$ in \eqref{SN}, we obtain \eqref{cbc-c} by a similar argument.
\end{proof}

If $\nabla_X \pmb{J}=0$ for all $X\in\mathcal{H}_{\mathbb{R}}$, then
the transformation rule \eqref{cbc-a} shows that $\mathcal{N}_{;\gamma}^\beta$ are  coefficients of a complex non-linear connection. In other words,  $\widetilde{\mathcal{H}}^{1,0}$ is a complex horizontal bundle, which is spanned by  $\{\frac{\delta}{\delta z^1},\cdots,\frac{\delta}{\delta z^n}\}$. Now we define
\begin{equation}
\delta v^\alpha=dv^\alpha+\mathcal{N}_{;\beta}^\alpha dz^\beta,\quad \alpha=1,\cdots,n.\label{hb}
\end{equation}
Then $\{\delta v^\alpha\}$ is a local frame for the dual bundle $(\mathcal{V}^{1,0})^\ast$. Thus we obtain the following decompositions
\begin{eqnarray*}
T_{\mathbb{C}}\tilde{M}&=&\widetilde{\mathcal{H}}_{\mathbb{C}}\oplus\mathcal{V}_{\mathbb{C}}=\widetilde{\mathcal{H}}^{1,0}\oplus\widetilde{\mathcal{H}}^{0,1}\oplus\mathcal{V}^{1,0}\oplus\mathcal{V}^{0,1},\\
T_{\mathbb{C}}^\ast\tilde{M}&=&\widetilde{\mathcal{H}}_{\mathbb{C}}^\ast\oplus\mathcal{V}_{\mathbb{C}}^\ast=(\widetilde{\mathcal{H}}^{1,0})^\ast\oplus(\widetilde{\mathcal{H}}^{0,1})^\ast\oplus(\mathcal{V}^{1,0})^\ast\oplus(\mathcal{V}^{0,1})^\ast.
\end{eqnarray*}

   Let $\Theta_{\mathbb{C}}:\mathcal{V}_{\mathbb{C}}\rightarrow \widetilde{\mathcal{H}}_{\mathbb{C}}$ be the complex horizontal map associated to $\widetilde{\mathcal{H}}_{\mathbb{C}}=\widetilde{\mathcal{H}}^{1,0}\oplus\widetilde{\mathcal{H}}^{0,1}$. Locally,
   $$\Theta_{\mathbb{C}}\left(\frac{\partial}{\partial v^\alpha}\right)=\frac{\delta}{\delta z^\alpha}\quad\text{and}\quad\Theta_{\mathbb{C}}\left(\frac{\partial}{\partial \bar{v}^\alpha}\right)=\frac{\delta}{\delta\bar{z}^\alpha},\quad \alpha=1,\cdots,n.$$
   Let $\hat{\Theta}:\mathcal{V}_{\mathbb{R}}\rightarrow \mathcal{H}_{\mathbb{R}}$ be the real horizontal map. Since $^o:\mathcal{V}^{1,0}\rightarrow\mathcal{V}_{\mathbb{R}}$ is an $\mathbb{R}$-isomorphism and $\widetilde{\mathcal{H}}^{1,0}$ is a complex horizontal bundle, we define an $\mathbb{R}$-isomorphism $^\sharp:\widetilde{\mathcal{H}}^{1,0}\rightarrow\mathcal{H}_{\mathbb{R}}$ by
\begin{equation}\label{eq-hrc}
H^\sharp=\hat{\Theta}((\Theta_{\mathbb{C}}^{-1}(H))^o),\quad \forall H\in\widetilde{\mathcal{H}}^{1,0}
\end{equation}
with  inverse  $_\sharp:\mathcal{H}_{\mathbb{R}}\rightarrow \widetilde{\mathcal{H}}^{1,0}$ given by
\begin{equation*}
H_\sharp=\Theta_{\mathbb{C}}((\hat{\Theta}^{-1}(H))_o),\quad \forall H\in\mathcal{H}_{\mathbb{R}}.
\end{equation*}

On the other hand,  $\mathcal{H}^{1,0}$ is a complex horizontal bundle, as in \cite{AP} we denote $\Theta:\mathcal{V}_{\mathbb{C}}\rightarrow \mathcal{H}_{\mathbb{C}}$ the complex horizontal map associated to $\mathcal{H}_{\mathbb{C}}=\mathcal{H}^{1,0}\oplus\mathcal{H}^{0,1}$. Then locally
$$\Theta\left(\frac{\partial}{\partial v^\alpha}\right)=\delta_\alpha\quad\text{and}\quad\Theta\left(\frac{\partial}{\partial \bar{v}^\alpha}\right)=\delta_{\bar{\alpha}},\quad \alpha=1,\cdots,n.$$
Similarly, we can define an $\mathbb{R}$-isomorphism $^o:\mathcal{H}^{1,0}\rightarrow\mathcal{H}_{\mathbb{R}}$ by
\begin{equation}\label{eq-hrc}
H^o=\hat{\Theta}((\Theta^{-1}(H))^o),\quad \forall H\in\mathcal{H}^{1,0}
\end{equation}
with  inverse  $_o:\mathcal{H}_{\mathbb{R}}\rightarrow \mathcal{H}^{1,0}$ given by
\begin{equation*}
H_o=\Theta((\hat{\Theta}^{-1}(H))_o),\quad \forall H\in\mathcal{H}_{\mathbb{R}}.
\end{equation*}

The following proposition shows that if $\pmb{J}$ is horizontal parallel with respect to $\nabla$, then $\pmb{J}$ commutes with $\hat{\Theta}$,   and $_\sharp=_o, ^\sharp=^o$.
\begin{proposition}Let $F: T^{1,0}M\rightarrow [0,+\infty)$ be a strongly convex complex Finsler metric on a complex manifold $M$.
 If $\nabla_X \pmb{J}=0$ for all $X\in\mathcal{H}_{\mathbb{R}}$, then
 $$ \hat{\Theta}\circ \pmb{J}=\pmb{J} \circ \hat{\Theta},\quad\text{and}\quad H^\sharp=H^o,\quad \forall H\in\widetilde{\mathcal{H}}^{1,0}.$$
\end{proposition}
\begin{proof}
Indeed,
\begin{eqnarray*}
&&\hat{\Theta}\left(\pmb{J}\frac{\partial}{\partial y^\alpha}\right) =	\hat{\Theta}\left(\frac{\partial}{\partial y^{\alpha^\ast   }}\right)=\frac{\delta}{\delta x^{\alpha^\ast   }}=\pmb{J}\left(	\hat{\Theta}\left(\frac{\partial}{\partial y^\alpha}\right)\right),\\
&&\hat{\Theta}\left(\pmb{J}\frac{\partial}{\partial y^{\alpha^\ast   }}\right) =	-\hat{\Theta}\left(\frac{\partial}{\partial y^{\alpha}}\right)=-\frac{\delta}{\delta x^{\alpha}}=\pmb{J}\left(	\hat{\Theta}\left(\frac{\partial}{\partial y^{\alpha^\ast   }}\right)\right).
	\end{eqnarray*}
Thus $\hat{\Theta}$ commutes with $\pmb{J}$.

Next, since
\begin{eqnarray*}
\left(\frac{\delta}{\delta z^\alpha}\right)^\sharp&=&\hat{\Theta}\left(\left(\Theta_{\mathbb{C}}^{-1}\left(\frac{\delta}{\delta z^\alpha}\right)\right)^o\right)=\hat{\Theta}\left(\left(\frac{\partial}{\partial v^\alpha}\right)^o\right)
=\hat{\Theta}\left(\frac{\partial}{\partial y^\alpha}\right)=\frac{\delta}{\delta x^\alpha},\\
\left(i\frac{\delta}{\delta z^\alpha}\right)^\sharp&=&\hat{\Theta}\left(\left(\Theta_{\mathbb{C}}^{-1}\left(i\frac{\delta}{\delta z^\alpha}\right)\right)^o\right)=\hat{\Theta}\left(\left(i\frac{\partial}{\partial v^\alpha}\right)^o\right)
=\hat{\Theta}\left(\frac{\partial}{\partial y^{\alpha^\ast   }}\right)
=\frac{\delta}{\delta x^{\alpha^\ast   }}.
\end{eqnarray*}
or equivalently
\begin{eqnarray*}
\left(\frac{\delta}{\delta x^\alpha}\right)_\sharp&=&\Theta_{\mathbb{C}}\left(\left(\hat{\Theta}^{-1}\left(\frac{\delta}{\delta x^\alpha}\right)\right)_o\right)=\Theta_{\mathbb{C}}\left(\left(\frac{\partial}{\partial y^\alpha}\right)_o\right)
=\Theta_{\mathbb{C}}\left(\frac{\partial}{\partial v^\alpha}\right)=\frac{\delta}{\delta z^\alpha},\\
\left(\frac{\delta}{\delta x^{\alpha^\ast   }}\right)_\sharp&=&\Theta_{\mathbb{C}}\left(\left(\hat{\Theta}^{-1}\left(\frac{\delta}{\delta x^{\alpha^\ast   }}\right)\right)_o\right)
=\Theta_{\mathbb{C}}\left(\left(\frac{\partial}{\partial y^{\alpha^\ast   }}\right)_o\right)=\Theta_{\mathbb{C}}\left(i\frac{\partial}{\partial v^\alpha}\right)=i\frac{\delta}{\delta z^\alpha}.
\end{eqnarray*}
Hence $^\sharp$ and $_\sharp$ are actually the restrictions of $^o$ and $_o$ to $\widetilde{\mathcal{H}}^{1,0}$ and $\mathcal{H}_{\mathbb{R}}$, respectively.
\end{proof}

\begin{proposition}\label{VT}
Let $F: T^{1,0}M\rightarrow [0,+\infty)$ be a strongly convex complex Finsler metric on a complex manifold $M$.
If $\nabla_X \pmb{J}=0$ for all $X\in\mathcal{H}_{\mathbb{R}}$, then
\begin{equation}
\frac{\partial }{\partial\bar{v}^\gamma}\mathcal{N}_{;\alpha}^\beta\equiv0,\quad \forall\alpha,\beta,\gamma=1,\cdots,n.
\end{equation}
\end{proposition}
\begin{proof}
Using \eqref{geodesic-2}, the equalities in \eqref{eq-cc} become
$$
\hat{\Gamma}_{\alpha}^\beta=\frac{\partial {\hat{\mathbb{G}}}^\beta}{\partial y^\alpha}=\frac{\partial {\hat{\mathbb{G}}}^{\beta^\ast }}{\partial y^{\alpha^\ast   }}=\hat{\Gamma}_{\alpha^\ast   }^{\beta^\ast  },\quad \hat{\Gamma}_{\alpha}^{\beta^\ast  }=\frac{\partial {\hat{\mathbb{G}}}^{\beta^\ast  }}{\partial y^\alpha}=-\frac{\partial {\hat{\mathbb{G}}}^\beta}{\partial y^{\alpha^\ast   }}=-\hat{\Gamma}_{\alpha^\ast   }^\beta,
$$
which implies
\begin{equation}
\frac{\partial}{\partial y^\alpha}(\hat{\mathbb{G}}^\beta+i\hat{\mathbb{G}}^{\beta^\ast  })=-i\frac{\partial}{\partial y^{\alpha^\ast   }}(\hat{\mathbb{G}}^\beta+i\hat{\mathbb{G}}^{\beta^\ast  }),\label{hnlc}
\end{equation}
or equivalently
\begin{equation}
 \frac{\partial}{\partial\bar{v}^\alpha}(\hat{\mathbb{G}}^\beta+i\hat{\mathbb{G}}^{\beta^\ast  })=0.\label{pa}
\end{equation}
From \eqref{hnlc}, we have
$$
\mathcal{N}_{;\alpha}^\beta=\hat{\Gamma}_{\alpha}^\beta+i\hat{\Gamma}_{\alpha}^{\beta^\ast  }=\frac{\partial}{\partial y^\alpha}(\hat{\mathbb{G}}^\beta+i\hat{\mathbb{G}}^{\beta^\ast  })=-i\frac{\partial}{\partial y^{\alpha^\ast   }}(\hat{\mathbb{G}}^\beta+i\hat{\mathbb{G}}^{\beta^\ast  }),
$$
so
\begin{equation}
\mathcal{N}_{;\alpha}^\beta=\frac{1}{2}\left(\frac{\partial}{\partial y^\alpha}-i\frac{\partial}{\partial y^{\alpha^\ast   }}\right)\left(\hat{\mathbb{G}}^\beta+i\hat{\mathbb{G}}^{\beta^\ast  }\right)=\frac{\partial}{\partial v^\alpha}(\hat{\mathbb{G}}^\beta+i\hat{\mathbb{G}}^{\beta^\ast  }).\label{pb}
\end{equation}
Then \eqref{pa} and \eqref{pb} give
 \begin{equation}
 \frac{\partial }{\partial\bar{v}^\gamma}\mathcal{N}_{;\alpha}^\beta=\frac{\partial^2}{\partial v^\alpha\partial\bar{v}^\gamma}(\hat{\mathbb{G}}^\beta+i\hat{\mathbb{G}}^{\beta^\ast  })\equiv0.
\end{equation}
\end{proof}

\begin{theorem}\label{Th-b}
Let $F: T^{1,0}M\rightarrow [0,+\infty)$ be a strongly convex complex Finsler metric on a complex manifold $M$. If  $\nabla\pmb{J}=0$,
then the extension $\nabla: \mathcal{X}(\mathcal{V}_{\mathbb{C}})\rightarrow \mathcal{X}(T_{\mathbb{C}}^\ast\tilde{M}\otimes \mathcal{V}_{\mathbb{C}})$ is a good complex vertical connection of type $(1,0)$,  and we have
\begin{eqnarray*}
\nabla_{\frac{\delta}{\delta z^\gamma}}\frac{\partial}{\partial v^\alpha}&=&\mathcal{N}_{\alpha;\gamma}^\beta\frac{\partial}{\partial v^\beta},\quad\nabla_{\frac{\delta}{\delta \bar{z}^\gamma}}\frac{\partial}{\partial v^\alpha}=0
\end{eqnarray*}
and
\begin{eqnarray*}
\nabla_{\frac{\partial}{\partial v^\gamma}}\frac{\partial}{\partial v^\alpha}	&=&\mathcal{N}_{\alpha\gamma}^\beta\frac{\partial}{\partial v^\beta},\quad
	 \nabla_{\frac{\partial}{\partial\bar{v}^\gamma}}\frac{\partial}{\partial v^\alpha}=0
\end{eqnarray*}
and its conjugations.
\end{theorem}

\begin{proof} Note that  $\nabla\pmb{J}=0$ implies that $\nabla_X \pmb{J}=0$ for any $X\in\mathcal{X}(\mathcal{H}_{\mathbb{R}})$. Thus by Corollary \ref{cor-a}, we have
\begin{equation*}
\hat{\Gamma}_{\alpha^\ast   }^\beta=-\hat{\Gamma}_{\alpha}^{\beta^\ast  },\quad \hat{\Gamma}_{\alpha^\ast   }^{\beta^\ast  }=\hat{\Gamma}_{\alpha}^\beta
\end{equation*}
for any $\alpha,\beta=1,\cdots,n$. This shows that  $\mathcal{H}_{\mathbb{R}}$ is $\pmb{J}$-invariant.

Extending $\nabla$ to $\mathcal{X}(\mathcal{V}_{\mathbb{C}})$ (still denote by $\nabla$) as follows:
\begin{equation}
\nabla_{X+iY}(V+iW):=\nabla_XV+i\nabla_XW+i\nabla_YV-\nabla_YW,\label{ECC}
\end{equation}
for any $X,Y\in T_{\mathbb{R}}\tilde{M}$ and $V,W\in\mathcal{X}(\mathcal{V}_{\mathbb{C}})$. Then by Proposition \ref{pro-b}, we have
 \begin{eqnarray*}
 \nabla_{\frac{\delta}{\delta z^\gamma}}\frac{\partial}{\partial v^\alpha}
&=&\frac{1}{4}\left(\hat{\Gamma}_{\alpha;\gamma}^\beta-i\hat{\Gamma}_{\alpha^\ast   ;\gamma}^\beta-i\hat{\Gamma}_{\alpha;\gamma^\ast }^\beta-\hat{\Gamma}_{\alpha^\ast   ;\gamma^\ast }^\beta\right)\frac{\partial}{\partial y^\beta}\\
&&+\frac{1}{4}\left(\hat{\Gamma}_{\alpha;\gamma}^{\beta^\ast  }-i\hat{\Gamma}_{\alpha^\ast   ;\gamma}^{\beta^\ast  }-i\hat{\Gamma}_{\alpha;\gamma^\ast }^{\beta^\ast  }-\hat{\Gamma}_{\alpha^\ast   ;\gamma^\ast }^{\beta^\ast  }\right)\frac{\partial}{\partial y^{\beta^\ast  }}.
\end{eqnarray*}
Using $\hat{\Gamma}^k_{j;l}=\hat{\Gamma}^k_{l;j}$ and \eqref{eq-ab}, it follows that
\begin{eqnarray*}
\left(\hat{\Gamma}_{\alpha;\gamma}^\beta\frac{\partial}{\partial y^\beta}-i\hat{\Gamma}_{\alpha^\ast   ;\gamma}^{\beta^\ast  }\frac{\partial}{\partial y^{\beta^\ast  }}\right)
&=&\left(\hat{\Gamma}_{\alpha;\gamma}^\beta\frac{\partial}{\partial y^\beta}-i\hat{\Gamma}_{\alpha;\gamma}^{\beta}\frac{\partial}{\partial y^{\beta^\ast  }}\right)=2\hat{\Gamma}_{\alpha;\gamma}^\beta\frac{\partial}{\partial v^\beta},\\
\left(-i\hat{\Gamma}_{\alpha^\ast   ;\gamma}^\beta\frac{\partial}{\partial y^\beta}+\hat{\Gamma}_{\alpha;\gamma}^{\beta^\ast  }\frac{\partial}{\partial y^{\beta^\ast  }}\right)
&=&\left(i\hat{\Gamma}_{\alpha;\gamma}^{\beta^\ast}\frac{\partial}{\partial y^\beta}+\hat{\Gamma}_{\alpha;\gamma}^{\beta^\ast  }\frac{\partial}{\partial y^{\beta^\ast  }}\right)
=2i\hat{\Gamma}_{\alpha;\gamma}^{\beta^\ast}\frac{\partial}{\partial v^\beta},\\
\left(-i\hat{\Gamma}_{\alpha;\gamma^\ast }^\beta\frac{\partial}{\partial y^\beta}-\hat{\Gamma}_{\alpha^\ast   ;\gamma^\ast }^{\beta^\ast  }\frac{\partial}{\partial y^{\beta^\ast  }}\right)
&=&\left(i\hat{\Gamma}_{\alpha;\gamma }^{\beta^\ast}+\hat{\Gamma}_{\gamma;\alpha}^{\beta^\ast}\frac{\partial}{\partial y^{\beta^\ast  }}\right)=2i\hat{\Gamma}_{\alpha;\gamma }^{\beta^\ast}\frac{\partial}{\partial v^\beta},\\
\left(-\hat{\Gamma}_{\alpha^\ast   ;\gamma^\ast }^\beta\frac{\partial}{\partial y^\beta}-i\hat{\Gamma}_{\alpha;\gamma^\ast }^{\beta^\ast  }\frac{\partial}{\partial y^{\beta^\ast  }}\right)
&=&\left(\hat{\Gamma}_{\alpha   ;\gamma^\ast }^{\beta^\ast}\frac{\partial}{\partial y^\beta}-i\hat{\Gamma}_{\gamma;\alpha}^{\beta  }\frac{\partial}{\partial y^{\beta^\ast  }}\right)
=2\hat{\Gamma}_{\gamma;\alpha}^{\beta  }\frac{\partial}{\partial v^\beta}.
\end{eqnarray*}
Thus
\begin{eqnarray*}
 \nabla_{\frac{\delta}{\delta z^\gamma}}\frac{\partial}{\partial v^\alpha}
&=&(\hat{\Gamma}_{\alpha;\gamma}^\beta+i\hat{\Gamma}_{\alpha;\gamma}^{\beta^\ast  })\frac{\partial}{\partial v^\beta}=\mathcal{N}_{\alpha;\gamma}^\beta\frac{\partial}{\partial v^\beta}.
\end{eqnarray*}
Similarly,
\begin{eqnarray*}
 \nabla_{\frac{\delta}{\delta \bar{z}^\gamma}}\frac{\partial}{\partial v^\alpha}
&=&\frac{1}{4}\left(\hat{\Gamma}_{\alpha;\gamma}^\beta-i\hat{\Gamma}_{\alpha^\ast   ;\gamma}^\beta+i\hat{\Gamma}_{\alpha;\gamma^\ast }^\beta+\hat{\Gamma}_{\alpha^\ast   ;\gamma^\ast }^\beta\right)\frac{\partial}{\partial y^\beta}\\
&&+\frac{1}{4}\left(\hat{\Gamma}_{\alpha;\gamma}^{\beta^\ast  }-i\hat{\Gamma}_{\alpha^\ast   ;\gamma}^{\beta^\ast  }+i\hat{\Gamma}_{\alpha;\gamma^\ast }^{\beta^\ast  }+\hat{\Gamma}_{\alpha^\ast   ;\gamma^\ast }^{\beta^\ast  }\right)\frac{\partial}{\partial y^{\beta^\ast  }}.
\end{eqnarray*}
 Using $\hat{\Gamma}^k_{j;l}=\hat{\Gamma}^k_{l;j}$ and \eqref{eq-ab} again, it follows that
 \begin{eqnarray*}
 \hat{\Gamma}_{\alpha;\gamma}^\beta+\hat{\Gamma}_{\alpha^\ast   ;\gamma^\ast }^\beta&=&0,\quad -\hat{\Gamma}_{\alpha^\ast   ;\gamma}^\beta+\hat{\Gamma}_{\alpha;\gamma^\ast }^\beta=0,\\
 \hat{\Gamma}_{\alpha;\gamma}^{\beta^\ast  }+\hat{\Gamma}_{\alpha^\ast   ;\gamma^\ast }^{\beta^\ast}&=&0,\quad-\hat{\Gamma}_{\alpha^\ast   ;\gamma}^{\beta^\ast  }+\hat{\Gamma}_{\alpha;\gamma^\ast }^{\beta^\ast  }=0.
 \end{eqnarray*}
 Thus
 $$
  \nabla_{\frac{\delta}{\delta \bar{z}^\gamma}}\frac{\partial}{\partial v^\alpha}=0.
 $$

By conjugation, we obtain
$$\nabla_{\frac{\delta}{\delta \bar{z}^\gamma}}\frac{\partial}{\partial \bar{v}^\alpha}=\overline{\mathcal{N}_{\alpha;\gamma}^\beta}\frac{\partial}{\partial \bar{v}^\beta},\quad
\nabla_{\frac{\delta}{\delta z^\gamma}}\frac{\partial}{\partial \bar{v}^\alpha}=0.$$

Similarly,
 \begin{eqnarray*}
	\nabla_{\frac{\partial}{\partial v^\gamma}}\frac{\partial}{\partial v^\alpha}
	&=&\frac{1}{4}\left(\hat{\Gamma}_{\alpha\gamma}^\beta-i\hat{\Gamma}_{\alpha^\ast   \gamma}^\beta-i\hat{\Gamma}_{\;\alpha\gamma^\ast }^\beta-\hat{\Gamma}_{\alpha^\ast   \gamma^\ast }^\beta\right)\frac{\partial}{\partial y^\beta}\\
	&&+\frac{1}{4}\left(\hat{\Gamma}_{\alpha\gamma}^{\beta^\ast  }-i\hat{\Gamma}_{\alpha^\ast   \gamma}^{\beta^\ast  }-i\hat{\Gamma}_{\alpha\gamma^\ast }^{\beta^\ast  }-\hat{\Gamma}_{\alpha^\ast   \gamma^\ast }^{\beta^\ast  }\right)\frac{\partial}{\partial y^{\beta^\ast  }}.
\end{eqnarray*}
Since $\nabla\pmb{J}=0$  also implies that $\nabla_X \pmb{J}=0$ for any $X \in\mathcal{V}_{\mathbb{R}}$, using  the equality $\hat{\Gamma}^k_{jl}=\hat{\Gamma}^k_{lj}$ and \eqref{eq-abc}, the above equality can be simplified as

\begin{eqnarray*}
	\nabla_{\frac{\partial}{\partial v^\gamma}}\frac{\partial}{\partial v^\alpha}
	&=&(\hat{\Gamma}_{\alpha\gamma}^{\beta}+i\hat{\Gamma}_{\alpha\gamma}^{\beta^\ast  })\frac{\partial}{\partial v^\beta}=\mathcal{N}_{\alpha\gamma}^\beta\frac{\partial}{\partial v^\beta}.
\end{eqnarray*}
Finally,
\begin{eqnarray*}
	\nabla_{\frac{\partial}{\partial \bar{v}^\gamma}}\frac{\partial}{\partial v^\alpha}
	&=&\frac{1}{4}\left(\hat{\Gamma}_{\alpha\gamma}^\beta-i\hat{\Gamma}_{\alpha^\ast   \gamma}^\beta+i\hat{\Gamma}_{\alpha\gamma^\ast }^\beta+\hat{\Gamma}_{\alpha^\ast   \gamma^\ast }^\beta\right)\frac{\partial}{\partial y^\beta}\\
	&&+\frac{1}{4}\left(\hat{\Gamma}_{\alpha\gamma}^{\beta^\ast  }-i\hat{\Gamma}_{\alpha^\ast   \gamma}^{\beta^\ast  }+i\hat{\Gamma}_{\alpha\gamma^\ast }^{\beta^\ast  }+\hat{\Gamma}_{\alpha^\ast   \gamma^\ast }^{\beta^\ast  }\right)\frac{\partial}{\partial y^{\beta^\ast  }}.
\end{eqnarray*}
Since by \eqref{eq-abc}, one can check that
\begin{eqnarray*}
\hat{\Gamma}_{\alpha\gamma}^\beta+\hat{\Gamma}_{\alpha^\ast   \gamma^\ast }^\beta&=&0,\quad-\hat{\Gamma}_{\alpha^\ast   \gamma}^\beta+\hat{\Gamma}_{\alpha\gamma^\ast }^\beta=0,\\
\hat{\Gamma}_{\alpha\gamma}^{\beta^\ast  }+\hat{\Gamma}_{\alpha^\ast   \gamma^\ast }^{\beta^\ast  }&=&0,\quad -\hat{\Gamma}_{\alpha^\ast   \gamma}^{\beta^\ast  }+\hat{\Gamma}_{\alpha\gamma^\ast }^{\beta^\ast  }=0.
\end{eqnarray*}
So that $$\nabla_{\frac{\partial}{\partial \bar{v}^\gamma}}\frac{\partial}{\partial v^\alpha}=0.$$
It is clear that
$$\nabla_{\frac{\partial}{\partial \bar{v}^\gamma}}\frac{\partial}{\partial \bar{v}^\alpha}	=\overline{\mathcal{N}_{\alpha\gamma}^\beta}\frac{\partial}{\partial \bar{v}^\beta}\quad\mbox{and}\quad \nabla_{\frac{\partial}{\partial v^\gamma}}\frac{\partial}{\partial \bar{v}^\alpha}=0.$$

Finally, that $\nabla:\mathcal{X}(\mathcal{V}^{1,0})\rightarrow \mathcal{X}(T_{\mathbb{C}}^\ast\tilde{M}\otimes\mathcal{V}^{1,0})$ is a good complex vertical connection follows immediately from Proposition \ref{prop3.3}. This completes the proof.
\end{proof}

\begin{theorem}\label{Th-c}
Let $F: T^{1,0}M\rightarrow [0,+\infty)$ be a strongly convex complex Finsler metric on a complex manifold $M$.  If $\nabla_X \pmb{J}=0$ for $X\in\mathcal{H}_{\mathbb{R}}$, then $F$ must be a K\"ahler-Berwald metric.
\end{theorem}

\begin{proof}
Let $V=\frac{1}{2}(X-i \pmb{J}X),W=\frac{1}{2}(Y-i \pmb{J}Y)$ with $X,Y\in\mathcal{V}_{\mathbb{R}}$. It follows that $V,W\in\mathcal{V}^{1,0}$ and  by Lemma \ref{L-a}, we have
\begin{equation}
\langle V,W\rangle=2\langle V\vert \overline{W}\rangle.\label{rc}
\end{equation}

By Theorem \ref{cartan}, the metric compatibility of the Cartan connection $\nabla$ implies that its natural extension  satisfies
\begin{equation}
Z\langle X\vert Y\rangle=\langle  \nabla_ZX\vert Y\rangle +\langle X\vert  \nabla_ZY\rangle\label{ETC-b}
\end{equation}
for any $Z=A+iB$ with $A, B\in T_{\mathbb{R}}\tilde{M}$ and $X,Y\in \mathcal{X}(\mathcal{V}_{\mathbb{C}})$.

By assumption, $\nabla_X\pmb{J}=0$ for all $X\in\mathcal{H}_{\mathbb{R}}$. Thus by Corollary \ref{cor-a}, $\mathcal{H}_{\mathbb{R}}$ is $\pmb{J}$-invariant. So that by Proposition \ref{pro-b}, $\frac{\delta}{\delta z^\alpha}=\frac{1}{2}(\frac{\delta}{\delta x^\alpha}-i\frac{\delta}{\delta x^{\alpha^\ast}})$.
Substituting $V=\frac{\partial}{\partial v^\beta}, W=\frac{\partial}{\partial v^\gamma}\in\mathcal{V}^{1,0}$ into \eqref{rc} and then using \eqref{ETC-b} with $Z=\frac{\delta}{\delta z^\alpha}=\frac{1}{2}(\frac{\delta}{\delta x^\alpha}-i\frac{\delta}{\delta x^{\alpha^\ast}})\in\widetilde{\mathcal{H}}^{1,0}$,  we obtain
\begin{eqnarray*}
\frac{\delta}{\delta z^\alpha}\left\langle \frac{\partial}{\partial v^\beta},\frac{\partial}{\partial v^\gamma}\right\rangle
&=&2\frac{\delta}{\delta z^\alpha}\left\langle \frac{\partial}{\partial v^\beta}\Big\vert \frac{\partial}{\partial \bar{v}^\gamma}\right\rangle\\
&=&2\left\langle \nabla_{\frac{\delta}{\delta z^\alpha}} \frac{\partial}{\partial v^\beta}\Big\vert \frac{\partial}{\partial\bar{v}^\gamma}\right\rangle+2\left\langle \frac{\partial}{\partial v^\beta}\Big\vert  \nabla_{\frac{\delta}{\delta z^\alpha}}\frac{\partial}{\partial\bar{v}^\gamma}\right\rangle\\
&=&\left\langle  \nabla_{\frac{\delta}{\delta z^\alpha}}\frac{\partial}{\partial v^\beta},\frac{\partial}{\partial v^\gamma}\right\rangle+\left\langle \frac{\partial}{\partial v^\alpha}, \nabla_{\frac{\delta}{\delta \bar{z}^\alpha}}\frac{\partial}{\partial v^\gamma}\right\rangle.
\end{eqnarray*}
This together with Theorem \ref{Th-b} imply
$$
\frac{\delta}{\delta z^\alpha} (G_{\beta\bar{\gamma}})=\left\langle\mathcal{N}_{\beta;\alpha}^\mu\frac{\partial}{\partial v^\mu},\frac{\partial}{\partial v^\gamma}\right\rangle=\mathcal{N}_{\beta;\alpha}^\mu G_{\mu\bar{\gamma}},
$$
from which it follows that
\begin{equation}
\mathcal{N}_{\beta;\alpha}^\mu=G^{\bar{\gamma}\mu}\frac{\delta}{\delta z^\alpha} (G_{\beta\bar{\gamma}})
=G^{\bar{\gamma}\mu}\left(\frac{\partial}{\partial z^\alpha}-\mathcal{N}_{;\alpha}^\nu\frac{\partial}{\partial v^\nu}\right)(G_{\beta\bar{\gamma}}).\label{ccc}
\end{equation}

Contracting both sides of \eqref{ccc} with respect to $v^{\beta}$, we obtain
$$
\mathcal{N}_{;\alpha}^\mu =G^{\bar{\gamma}\mu}\frac{\partial}{\partial z^\alpha}(G_{\beta\bar{\gamma}})v^\beta=G^{\bar{\gamma}\mu}\frac{\partial^2F^2}{\partial z^\alpha\partial\bar{v}^\gamma}=\Gamma_{;\alpha}^\mu,
$$
where we use the equality $\frac{\partial}{\partial v ^{\nu}}(G_{\beta\bar{\gamma}})v^\beta=0$ and $\Gamma_{;\alpha}^\mu$ are  the complex non-linear connection coefficients of $F$ defined by \eqref{HVB}.
Substituting $\Gamma_{;\alpha}^\mu=\mathcal{N}_{;\alpha}^\mu$ into \eqref{ccc}, we obtain
\begin{equation}
\mathcal{N}_{\beta;\alpha}^\mu=G^{\bar{\gamma}\mu}\left(\frac{\partial}{\partial z^\alpha}-\Gamma_{;\alpha}^\nu\frac{\partial}{\partial v ^{\nu}}\right)\left(G_{\beta\bar{\gamma}}\right)=\Gamma_{\beta;\alpha}^\mu,
\end{equation}
where $\Gamma_{\beta;\alpha}^\mu$ are exactly the horizontal connection coefficients of the Chern-Finsler connection given by \eqref{HVC}. By Remark \ref{R-a}, we alread have $\mathcal{N}_{\beta;\alpha}^\mu=\mathcal{N}_{\alpha;\beta}^\mu$, hence $\Gamma_{\beta;\alpha}^\mu =\Gamma_{\alpha;\beta}^\mu$. That is,  $F$ must be a K\"ahler-Finsler metric.

Next we show that $F$ actually must be a K\"ahler-Berwald metric.  By Proposition \ref{VT} and the quality $\Gamma_{;\alpha}^\mu=\mathcal{N}_{;\alpha}^\mu$,  we have
\begin{eqnarray*}
	0=	\frac{\partial}{\partial \bar{v}^{\beta}}\Gamma_{;\alpha}^{\mu}&=&\frac{1}{2}\left(\frac{\partial}{\partial y^{\beta}}+i\frac{\partial}{\partial y^{\beta^\ast  }}\right) \Gamma_{;\alpha}^{\mu} \\
	&=&\frac{1}{2}\left(\frac{\partial}{\partial y^{\beta}}+i\frac{\partial}{\partial y^{\beta^\ast  }}\right)\left(\hat{\Gamma}_{\alpha}^{\mu}+i\hat{\Gamma}_{\alpha}^{\mu^\ast  }\right)\\
	&=&\frac{1}{2}\left[(\hat{\mathbb{G}}_{\beta\alpha}^{\mu}-\hat{\mathbb{G}}_{\beta^\ast  \,\alpha}^{\mu^\ast  })+i(\hat{\mathbb{G}}_{\beta\alpha}^{\mu^\ast  }+\hat{\mathbb{G}}_{\beta^\ast  \,\alpha}^{\mu}) \right],
\end{eqnarray*}
where $\hat{\mathbb{G}}^{k}_{jl}=\frac{\partial ^2 \hat{\mathbb{G}}^k}{\partial y^j \partial y^l}$ are the real Berwald connection coefficients associated to $F$.
This implies that \begin{equation}\label{B-a}
	\hat{\mathbb{G}}_{\beta\alpha}^{\mu}=\hat{\mathbb{G}}_{\beta^\ast  \,\alpha}^{\mu^\ast  },\quad \hat{\mathbb{G}}_{\beta\alpha}^{\mu^\ast  }=-\hat{\mathbb{G}}_{\beta^\ast  \,\alpha}^{\mu}.
\end{equation}
Thus we have
\begin{eqnarray*}
	\frac{\partial }{\partial \bar{v}^\gamma}\Gamma^{\mu}_{\beta;\alpha}=\frac{\partial^2}{\partial \bar{v}^\gamma\partial v^{\beta}}\Gamma^{\mu}_{;\alpha}=\frac{\partial^2}{\partial v^{\beta}\partial \bar{v}^\gamma}\Gamma^{\mu}_{;\alpha}=0.
	\end{eqnarray*}
A direct computation shows
\begin{eqnarray*}
\Gamma_{\beta;\alpha}^{\mu}=\frac{\partial}{\partial v^{\beta}}\Gamma_{;\alpha}^{\mu}&=&\frac{\partial}{\partial v^{\beta}}\mathcal{N}_{;\alpha}^{\mu}\\
&=&\frac{1}{2}\left(\frac{\partial}{\partial y^{\beta}}-i\frac{\partial}{\partial y^{\beta^\ast  }}\right) \mathcal{N}_{;\alpha}^{\mu} \\
&=&\frac{1}{2}\left(\frac{\partial}{\partial y^{\beta}}-i\frac{\partial}{\partial y^{\beta^\ast  }}\right)\left(\hat{\Gamma}_{\alpha}^{\mu}+i\hat{\Gamma}_{\alpha}^{\mu^\ast  }\right)\\
&=&\frac{1}{2}\left[\hat{\mathbb{G}}_{\beta\alpha}^{\mu}+i\hat{\mathbb{G}}_{\beta\alpha}^{\mu^\ast  }-i\hat{\mathbb{G}}_{\beta^\ast  \,\alpha}^{\mu}+\hat{\mathbb{G}}_{\beta^\ast  \,\alpha}^{\mu^\ast  } \right]\\
&=&\hat{\mathbb{G}}_{\beta\alpha}^{\mu}+i\hat{\mathbb{G}}_{\beta\alpha}^{\mu^\ast  },
\end{eqnarray*}
where we used \eqref{B-a} in the last equality.
On the other hand, we have
$$\Gamma_{\beta;\alpha}^\mu=\mathcal{N}_{\beta;\alpha}^\mu=\hat{\Gamma}^{\mu}_{\beta;\alpha}+i\hat{\Gamma}^{\mu^\ast  }_{\beta;\alpha}.$$
Thus we must have
$$\hat{\Gamma}^{\mu}_{\beta;\alpha}=\hat{\mathbb{G}}^{\mu}_{\beta\alpha},\quad \text{and}\quad
\hat{\Gamma}^{\mu^\ast  }_{\beta;\alpha}=\hat{\mathbb{G}}^{\mu^\ast  }_{\beta\alpha}.$$
This implies that real Berwald connection coefficients coincide with the horizontal connection coefficients of the Cartan connection associated to $F$. Since the horizontal Cartan connection coefficients (the same as the Chern connection coefficients) and the real Berwald connection coefficients satisfy $\hat{\Gamma}^{k}_{j;l}=\hat{\mathbb{G}}^{k}_{jl}-\dot{A}^k_{\;jl}$ (see p. 39 in \cite{BCS}), it follows that $\dot{A}_{\;jl}^k=0$, hence $F$ must be a real Landsberg metric. Furthermore, a strongly convex weakly K\"ahler-Finsler metric is a real Landsberg metric iff it is a weakly complex Berwald metric \cite{he-zhong}, and we have already showed that $F$ is a K\"ahler-Finsler metric, it follows that $F$ must be  a K\"ahler-Berwald metric.
\end{proof}

\begin{theorem}\label{Th-cc}
Let $F: T^{1,0}M\rightarrow [0,+\infty)$ be a strongly convex complex Finsler metric on a complex manifold $M$.  Then $\nabla \pmb{J}=0$ iff the Cartan connection $\nabla$ associated to $F$ coincides with the Chern-Finsler connection $D$ associated to $F$.
\end{theorem}

\begin{proof} First we show the necessity.  Suppose $\nabla \pmb{J}=0$. Extending $\nabla$ to $\mathcal{V}_{\mathbb{C}}$, then by Theorem \ref{Th-b}, we obtain a good complex vertical connection $\nabla:\mathcal{X}(\mathcal{V}^{1,0})\rightarrow\mathcal{X}(T_{\mathbb{C}}^\ast\tilde{M}\otimes \mathcal{V}^{1,0})$, which is of type $(1,0)$  such that
$$
\nabla_{\frac{\delta}{\delta z^\gamma}}\frac{\partial}{\partial v^\alpha}=\mathcal{N}_{\alpha;\gamma}^\beta\frac{\partial}{\partial v^\beta},\quad \nabla_{\frac{\partial}{\partial v^\gamma}}\frac{\partial}{\partial v^\alpha}=\mathcal{N}_{\alpha\gamma}^\beta\frac{\partial}{\partial v^\beta}.
$$
Since $\nabla\pmb{J}=0$ implies $\nabla_X\pmb{J}=0$ for all $X\in\mathcal{H}_{\mathbb{R}}$, thus by Theorem \ref{Th-c}, we have $\mathcal{N}_{\alpha;\gamma}^\beta=\Gamma_{\alpha;\gamma}^\beta$. Next we show $\mathcal{N}_{\alpha\gamma}^\beta=\Gamma_{\alpha\gamma}^\beta$.
Using \eqref{rc} and \eqref{ETC-b}, we have
\begin{eqnarray*}
\frac{\partial}{\partial v^\alpha}\left\langle \frac{\partial}{\partial v^\beta},\frac{\partial}{\partial v^\gamma}\right\rangle
&=&2\frac{\partial}{\partial v^\alpha}\left\langle\frac{\partial}{\partial v^\beta}\Big\vert \frac{\partial}{\partial \bar{v}^\gamma}\right\rangle\\
&=&2\left\langle\nabla_{\frac{\partial}{\partial v^\alpha}}\frac{\partial}{\partial v^\beta}\Big\vert \frac{\partial}{\partial \bar{v}^\gamma}\right\rangle
+\left\langle\frac{\partial}{\partial v^\beta}\Big\vert \nabla_{\frac{\partial}{\partial v^\alpha}}\frac{\partial}{\partial \bar{v}^\gamma}\right\rangle\\
&=&\left\langle \nabla_{\frac{\partial}{\partial v^\alpha}}\frac{\partial}{\partial v^\beta},\frac{\partial}{\partial v^\gamma}\right\rangle+\left\langle \frac{\partial}{\partial v^\alpha}, \nabla_{\frac{\partial}{\partial \bar{v}^\alpha}}\frac{\partial}{\partial v^\gamma}\right\rangle.
\end{eqnarray*}
Again using Theorem \ref{Th-b}, the above equality reduces to
$$
\frac{\partial}{\partial v^\alpha} (G_{\beta\bar{\gamma}})=\left\langle\mathcal{N}_{\beta\alpha}^\mu\frac{\partial}{\partial v^\mu},\frac{\partial}{\partial v^\gamma}\right\rangle=\mathcal{N}_{\beta\alpha}^\mu G_{\mu\bar{\gamma}}.
$$
Hence,
\begin{equation}
	\mathcal{N}_{\beta\alpha}^\mu=G^{\bar{\gamma}\mu}\frac{\partial}{\partial v^\alpha} (G_{\beta\bar{\gamma}})
	=\Gamma_{\beta\alpha}^\mu,
\end{equation}
which are exactly the vertical connection coefficients of the Chern-Finsler connection associated to $F$. Thus $\nabla:\mathcal{X}(\mathcal{V}^{1,0})\rightarrow \mathcal{X}(T_{\mathbb{C}}^\ast \tilde{M}\otimes \mathcal{V}^{1,0})$ coincides with the Chern-Finsler connection $D$ associated to $F$.

Next we show the sufficiency. By Theorem \ref{Th-ab}, it suffices to show that \eqref{eq-ab} and \eqref{eq-abc} hold.
By assumption, the Cartan connection $\nabla$ and the Chern-Finsler connection $D$ coincide, thus the real non-linear connection $\tilde{\nabla}$ associated to $\nabla$ and the complex non-linear connection $\tilde{D}$ associated to $D$ coincide. In other words, let $\xi\in\mathcal{X}(T^{1,0}M)$ be a holomorphic vector field and $w\in T_p^{1,0}M$ such that $\eta_o=\xi$ and $u_o=w$, then at the point $p$, we have
\begin{equation}
(\tilde{D}_w\xi)(p)=((\tilde{\nabla}_u\eta)_o)(p),\label{RAC}
\end{equation}
%Writing $\xi=\xi^\alpha\frac{\partial}{\partial z^\alpha},\eta=\eta^k\frac{\partial}{\partial x^k}$ and $w=w^\alpha\frac{\partial}{\partial z^\alpha}\vert _p, u=u^k\frac{\partial}{\partial x^k}$ for $\xi^\alpha=\eta^\alpha+i\eta^{\alpha^\ast}$ and $w^\alpha=u^\alpha+iu^{\alpha^\ast}$, then it is easy to check that
which implies
\begin{equation}
\hat{\Gamma}_\alpha^\beta=\hat{\Gamma}_{\alpha^\ast}^{\beta^\ast},\quad \hat{\Gamma}_\alpha^{\beta^\ast}=-\hat{\Gamma}_{\alpha^\ast}^\beta\label{A-a}
\end{equation}
and
\begin{equation}
\mathcal{N}_{;\alpha}^\beta=\Gamma_{;\alpha}^\beta=\hat{\Gamma}_\alpha^\beta+i\hat{\Gamma}_\alpha^{\beta^\ast}.
\end{equation}
Note that \eqref{A-a} implies that $\mathcal{H}_{\mathbb{R}}$ is $\pmb{J}$-invariant, so that the complex horizontal bundle $\widetilde{\mathcal{H}}^{1,0}$ spanned by $\{\frac{\delta}{\delta z^1},\cdots,\frac{\delta}{\delta z^n}\}$ actually coincides with $\mathcal{H}^{1,0}$, namely $\frac{\delta}{\delta z^\alpha}=\delta_\alpha$ for $\alpha=1,\cdots,n$. Thus by Proposition \ref{pro-b}, we have
\begin{equation}
\left(\frac{\delta}{\delta x^\alpha}\right)_o=\frac{\delta}{\delta z^\alpha}=\delta_\alpha\quad\text{and}\quad \left(\frac{\delta}{\delta x^{\alpha^\ast   }}\right)_o=i\frac{\delta}{\delta z^\alpha}=i\delta_\alpha.
\end{equation}
Therefore
\begin{eqnarray*}
0&=&D_{\delta_{\bar{\gamma}}}\frac{\partial}{\partial v^\alpha}=	\nabla_{\frac{\delta}{\delta \bar{z}^\gamma}}\frac{\partial}{\partial v^\alpha}\\
	&=&\frac{1}{2}\left(\nabla_{\frac{\delta}{\delta x^\gamma}}\frac{\partial}{\partial v^\alpha}+i\nabla_{\frac{\delta}{\delta x^{\gamma^\ast }}}\frac{\partial}{\partial v^\alpha}\right)\\
	&=&\frac{1}{4}\left(\nabla_{\frac{\delta}{\delta x^\gamma}}\frac{\partial}{\partial y^\alpha}-i\nabla_{\frac{\delta}{\delta x^\gamma}}\frac{\partial}{\partial y^{\alpha^\ast   }}
	+i\left(\nabla_{\frac{\delta}{\delta x^{\gamma^\ast }}}\frac{\partial}{\partial y^\alpha}-i\nabla_{\frac{\delta}{\delta x^{\gamma^\ast }}}\frac{\partial}{\partial y^{\alpha^\ast   }}\right)\right)\\
%	&=&\frac{1}{4}\left(\hat{\Gamma}_{\alpha;\gamma}^l-i\hat{\Gamma}_{\alpha^\ast   ;\gamma}^l+i\hat{\Gamma}_{\alpha;\gamma^\ast }^l+\hat{\Gamma}_{\alpha^\ast   ;\gamma^\ast }^l\right)\frac{\partial}{\partial y^l}\\
	&=&\frac{1}{4}\left(\hat{\Gamma}_{\alpha;\gamma}^\beta-i\hat{\Gamma}_{\alpha^\ast   ;\gamma}^\beta+i\hat{\Gamma}_{\alpha;\gamma^\ast }^\beta+\hat{\Gamma}_{\alpha^\ast   ;\gamma^\ast }^\beta\right)\frac{\partial}{\partial y^\beta}\\
	&&+\frac{1}{4}\left(\hat{\Gamma}_{\alpha;\gamma}^{\beta^\ast  }-i\hat{\Gamma}_{\alpha^\ast   ;\gamma}^{\beta^\ast  }+i\hat{\Gamma}_{\alpha;\gamma^\ast }^{\beta^\ast  }+\hat{\Gamma}_{\alpha^\ast   ;\gamma^\ast }^{\beta^\ast  }\right)\frac{\partial}{\partial y^{\beta^\ast  }},
\end{eqnarray*}
which means that
 \begin{eqnarray}
	\hat{\Gamma}_{\alpha;\gamma}^\beta=-\hat{\Gamma}_{\alpha^\ast   ;\gamma^\ast }^\beta,\quad \hat{\Gamma}_{\alpha^\ast   ;\gamma}^\beta=\hat{\Gamma}_{\alpha;\gamma^\ast }^\beta,\label{eq-CC1}\\
	\hat{\Gamma}_{\alpha;\gamma}^{\beta^\ast  }=-\hat{\Gamma}_{\alpha^\ast   ;\gamma^\ast }^{\beta^\ast  },\quad \hat{\Gamma}_{\alpha^\ast   ;\gamma}^{\beta^\ast  }=\hat{\Gamma}_{\alpha;\gamma^\ast }^{\beta^\ast  }.\label{eq-CC2}
\end{eqnarray}
Using \eqref{eq-CC1} and \eqref{eq-CC2}, we have
 \begin{eqnarray*}
\Gamma^\beta_{\alpha;\gamma}\frac{\partial}{\partial v^\beta}&=&D_{\delta_\gamma}\frac{\partial}{\partial v^\alpha}=\nabla_{\frac{\delta}{\delta z^\gamma}}\frac{\partial}{\partial v^\alpha}\\
	%&=&\frac{1}{2}\left(\nabla_{\frac{\delta}{\delta x^\gamma}}\frac{\partial}{\partial v^\alpha}-i\nabla_{\frac{\delta}{\delta x^{\gamma^\ast }}}\frac{\partial}{\partial v^\alpha}\right)\\
	&=&\frac{1}{4}\left(\nabla_{\frac{\delta}{\delta x^\gamma}}\frac{\partial}{\partial y^\alpha}-i\nabla_{\frac{\delta}{\delta x^\gamma}}\frac{\partial}{\partial y^{\alpha^\ast   }}
	-i\nabla_{\frac{\delta}{\delta x^{\gamma^\ast }}}\frac{\partial}{\partial y^\alpha}-\nabla_{\frac{\delta}{\delta x^{\gamma^\ast }}}\frac{\partial}{\partial y^{\alpha^\ast   }}\right)\\
%	&=&\frac{1}{4}\left(\hat{\Gamma}_{\alpha;\gamma}^l-i\hat{\Gamma}_{\alpha^\ast   ;\gamma}^l-i\hat{\Gamma}_{\alpha;\gamma^\ast }^l-\hat{\Gamma}_{\alpha^\ast   ;\gamma^\ast }^l\right)\frac{\partial}{\partial y^l}\\
	&=&\frac{1}{4}\left(\hat{\Gamma}_{\alpha;\gamma}^\beta-i\hat{\Gamma}_{\alpha^\ast   ;\gamma}^\beta-i\hat{\Gamma}_{\alpha;\gamma^\ast }^\beta-\hat{\Gamma}_{\alpha^\ast   ;\gamma^\ast }^\beta\right)\frac{\partial}{\partial y^\beta}\\
	&&+\frac{1}{4}\left(\hat{\Gamma}_{\alpha;\gamma}^{\beta^\ast  }-i\hat{\Gamma}_{\alpha^\ast   ;\gamma}^{\beta^\ast  }-i\hat{\Gamma}_{\alpha;\gamma^\ast }^{\beta^\ast  }-\hat{\Gamma}_{\alpha^\ast   ;\gamma^\ast }^{\beta^\ast  }\right)\frac{\partial}{\partial y^{\beta^\ast  }}\\
	&=&\frac{1}{2}\left(\hat{\Gamma}_{\alpha;\gamma}^\beta-i\hat{\Gamma}_{\alpha^\ast   ;\gamma}^\beta\right)\frac{\partial}{\partial y^\beta}+\frac{1}{2}\left(\hat{\Gamma}_{\alpha;\gamma}^{\beta^\ast  }-i\hat{\Gamma}_{\alpha^\ast   ;\gamma}^{\beta^\ast  }\right)\frac{\partial}{\partial y^{\beta^\ast  }},
\end{eqnarray*}
which implies that
$$\text{Re}\Gamma_{\alpha;\gamma}^\beta=\hat{\Gamma}_{\alpha;\gamma}^\beta= \hat{\Gamma}_{\alpha^\ast   ;\gamma}^{\beta^\ast  },\quad \text{Im}\Gamma_{\alpha;\gamma}^\beta=-\hat{\Gamma}_{\alpha^\ast   ;\gamma}^\mu=\hat{\Gamma}_{\alpha;\gamma}^{\beta^\ast  }.
$$
Thus we obtain \eqref{eq-ab}.

Similarly, we have
\begin{eqnarray*}
0&=&D_{\frac{\partial}{\partial \bar{v}^\gamma}}\frac{\partial}{\partial v^\alpha}=\nabla_{\frac{\partial}{\partial \bar{v}^\gamma}}\frac{\partial}{\partial v^\alpha}\\
&=&\frac{1}{4}(\hat{\Gamma}_{\alpha\gamma}^\beta-i\hat{\Gamma}_{\alpha^\ast   \gamma}^\beta+i\hat{\Gamma}_{\alpha\gamma^\ast }^\beta+\hat{\Gamma}_{\alpha^\ast   \gamma^\ast }^\beta)\frac{\partial}{\partial y^\beta}\\
&&	+\frac{1}{4}(\hat{\Gamma}_{\alpha\gamma}^{\beta^\ast  }-i\hat{\Gamma}_{\alpha^\ast   \gamma}^{\beta^\ast  }+i\hat{\Gamma}_{\alpha\gamma^\ast }^{\beta^\ast  }+\hat{\Gamma}_{\alpha^\ast   \gamma^\ast }^{\beta^\ast  })\frac{\partial}{\partial y^{\beta^\ast  }}.
\end{eqnarray*}
It follows that
\begin{eqnarray}
		\hat{\Gamma}_{\alpha\gamma}^\beta=-\hat{\Gamma}_{\alpha^\ast   \gamma^\ast }^\beta,\quad \hat{\Gamma}_{\alpha^\ast   \gamma}^\beta=\hat{\Gamma}_{\alpha\gamma^\ast }^\beta,\quad
		\hat{\Gamma}_{\alpha\gamma}^{\beta^\ast  }=-\hat{\Gamma}_{\alpha^\ast \gamma^\ast }^{\beta^\ast  },\quad \hat{\Gamma}_{\alpha^\ast \gamma}^{\beta^\ast}
		=\hat{\Gamma}_{\alpha\gamma^\ast }^{\beta^\ast  }.\label{B-a}
\end{eqnarray}
On the other hand,
\begin{eqnarray*}
\Gamma_{\alpha\gamma}^\beta\frac{\partial}{\partial v^\beta}&=&D_{\frac{\partial}{\partial v^\gamma}}\frac{\partial}{\partial v^\alpha}=\nabla_{\frac{\partial}{\partial v^\gamma}}\frac{\partial}{\partial v^\alpha}\\
&=&\frac{1}{4}(\hat{\Gamma}_{\alpha\gamma}^\beta-i\hat{\Gamma}_{\alpha^\ast   \gamma}^\beta-i\hat{\Gamma}_{\;\alpha\gamma^\ast }^\beta-\hat{\Gamma}_{\alpha^\ast   \gamma^\ast }^\beta)\frac{\partial}{\partial y^\beta}\\
&&	+\frac{1}{4}(\hat{\Gamma}_{\alpha\gamma}^{\beta^\ast  }-i\hat{\Gamma}_{\alpha^\ast   \gamma}^{\beta^\ast  }-i\hat{\Gamma}_{\alpha\gamma^\ast }^{\beta^\ast  }-\hat{\Gamma}_{\alpha^\ast   \gamma^\ast }^{\beta^\ast  })\frac{\partial}{\partial y^{\beta^\ast  }},
\end{eqnarray*}
which together with \eqref{B-a} implies that
\begin{equation}
	\text{Re}\Gamma_{\alpha\gamma}^\beta=\hat{\Gamma}_{\alpha\gamma}^\beta= \hat{\Gamma}_{\alpha^\ast   \gamma}^{\beta^\ast  },\quad \text{Im}\Gamma_{\alpha\gamma}^\beta=-\hat{\Gamma}_{\alpha^\ast   \gamma}^\beta=\hat{\Gamma}_{\alpha\gamma}^{\beta^\ast  }.
	\end{equation}
Thus we obtain \eqref{eq-abc}. This completes the proof.
\end{proof}

\begin{theorem}\label{Th-d}
Let $F: T^{1,0}M\rightarrow [0,+\infty)$ be a strongly convex K\"ahler-Berwald metric on a complex manifold $M$.
Then  $\nabla_X \pmb{J}=0$ for all $X\in\mathcal{H}_{\mathbb{R}}$.
\end{theorem}
\begin{proof}
Since $F$ is a K\"ahler-Berwald metric, by Lemma 3.1 in \cite{XZ}, $\mathbb{G}^\mu$ are holomorphic with respect to $v$ and $\mathbb{G}^\mu=\hat{\mathbb{G}}^\mu+i\hat{\mathbb{G}}^{\mu^\ast}$. Thus
\begin{eqnarray*}
0=\frac{\partial\mathbb{G}^\mu}{\partial \bar{v}^{\alpha}}
%&=&\frac{1}{2}\left(\frac{\partial}{\partial y^{\beta}}+i\frac{\partial}{\partial y^{\beta^\ast  }}\right)\mathcal{N}_{;\alpha}^{\mu} \\
	&=&\frac{1}{2}\left(\frac{\partial}{\partial y^{\alpha}}+i\frac{\partial}{\partial y^{\alpha^\ast  }}\right)\left(\hat{\mathbb{G}}^\mu+i\hat{\mathbb{G}}^{\mu^\ast}\right)\\
	&=&\frac{1}{2}\left(\hat{\Gamma}_\alpha^\mu+i\hat{\Gamma}_\alpha^{\mu^\ast}+i\hat{\Gamma}_{\alpha^\ast}^\mu-\hat{\Gamma}_{\alpha^\ast}^{\mu^\ast}\right),
	\end{eqnarray*}
from which we get
\begin{equation}
\hat{\Gamma}_\alpha^\mu=\hat{\Gamma}_{\alpha^\ast}^{\mu^\ast},\quad\hat{\Gamma}_\alpha^{\mu^\ast}=-\hat{\Gamma}_{\alpha^\ast}^\mu.\label{337}
\end{equation}
Differentiating \eqref{337}, we obtain
$$
\hat{\mathbb{G}}_{j\alpha}^\mu=\hat{\mathbb{G}}_{j\alpha^\ast}^{\mu^\ast},\quad \hat{\mathbb{G}}_{j\alpha}^{\mu^\ast}=-\hat{\mathbb{G}}_{j\alpha^\ast}^\mu.
$$

By Theorem 1.2 in \cite{Zhong-a}, a strongly convex K\"ahler-Berwald metric is necessary a real Berwald metric, hence a real Landsberg metric \cite{Shen}. So   $\dot{A}^k_{\;jl}=0$,  which implies $\hat{\mathbb{G}}^k_{jl}=\hat{\Gamma}^k_{j;l}$. That is, the equalities in \eqref{eq-ab} hold. Thus $\nabla_X \pmb{J}=0$ for all $X\in\mathcal{H}_{\mathbb{R}}$.
\end{proof}

Let $\Theta:\mathcal{V}^{1,0}\rightarrow\mathcal{H}^{1,0}$ denote the complex horizontal map associated to $\mathcal{H}^{1,0}$, and let $\hat{\Theta}:\mathcal{V}_{\mathbb{R}}\rightarrow \mathcal{H}_{\mathbb{R}}$ denote the real horizontal map associated to $\mathcal{H}_{\mathbb{R}}$. Since $^o:\mathcal{V}^{1,0}\rightarrow\mathcal{V}_{\mathbb{R}}$ is a $\mathbb{R}$-isomorphism, we get a $\mathbb{R}$-isomorphism $^{\hat{}}:\mathcal{H}^{1,0}\rightarrow\mathcal{H}_{\mathbb{R}}$ given by
$$
\hat{H}=\hat{\Theta}((\Theta^{-1}(H))^o),\quad \forall H\in \mathcal{H}^{1,0}.
$$
Note that $\widehat{\delta_\alpha}=\frac{\delta}{\delta x^\alpha}$ and $\widehat{i\delta_\alpha}=\frac{\delta}{\delta x^{\alpha^\ast}}$.

Let
$$\hat{\chi}=y^j\frac{\delta}{\delta x^j}=y^j\left(\frac{\partial}{\partial x^j}-\hat{\Gamma}_j^k\frac{\partial}{\partial y^k}\right)\quad\mbox{and}\quad\chi=v^\alpha\left(\frac{\partial}{\partial z^\alpha}-\Gamma_{;\alpha}^\beta\frac{\partial}{\partial v^\beta}\right)$$
 denote the real and complex radial horizontal vector fields associated to $F$, respectively. Let $\hat{\Omega}$ and $\Omega$ denote the curvature operators of the Cartan connection $\nabla$ and the Chern-Finsler connection $D$ associated to $F$, respectively.

 Using some known results established by Abate and Patrizio in section 2.6 in \cite{AP}, we are now able to obtain the following theorem which shows the specialities of strongly convex K\"ahler-Berwald metrics.

\begin{theorem}\label{SE}
Let $F: T^{1,0}M\rightarrow[0,+\infty)$ be a strongly convex K\"ahler-Berwald metric on a complex manifold $M$. Then for all $H\in\mathcal{H}^{1,0}$ and $V\in\mathcal{V}^{1,0}$, we have

(1) $\chi^o=\hat{\chi}$;

(2) $\hat{\Theta}$ commutes with  $ \pmb{J}$;

(3) $(\pmb{J}H)^o= \pmb{J}H^o$;

(4) $\mathcal{H}_{\mathbb{R}}$ is $ \pmb{J}$-invariant;

(5) $H=(H^o)_o$;

(6) $\Gamma_{;\alpha}^\beta=\hat{\Gamma}_{\alpha}^\beta+i\hat{\Gamma}_{\alpha}^{\beta^\ast  }$;

(7) $(\frac{\delta}{\delta x^\alpha})_o=\frac{\delta}{\delta z^\alpha}$ and $(\frac{\delta}{\delta x^{\alpha^\ast   }})_o=i\frac{\delta}{\delta z^\alpha}$;

(8)
$$
\hat{\Gamma}_a^b=\left\{
                      \begin{array}{ll}
                        \text{Re}\,\Gamma_{;\alpha}^\beta, & \hbox{if}\;\;1\leq a,b\leq n, \\
                        \text{Im}\,\Gamma_{;\alpha}^\beta, & \hbox{if}\;\;1\leq a\leq n\;\;\text{and}\;\;n+1\leq b\leq 2n, \\
                        -\text{Im}\,\Gamma_{;\alpha}^\beta, & \hbox{if}\;\;n+1\leq a\leq 2n\;\;\text{and}\;\;1\leq b\leq n, \\
                        \text{Re}\,\Gamma_{;\alpha}^\beta, & \hbox{if}\;\;n+1\leq a,b\leq 2n.
                      \end{array}
                    \right.
$$

(9) $\nabla_{\hat{\chi}}V^o=(\nabla_{\chi^o}V)^o$;

(10) $\nabla_{\hat{H}}V^o=(\nabla_{H^o}V)^o$;

(11) $\langle\hat{\Omega}(\hat{\chi},\hat{H})\hat{H}\vert \hat{\chi}\rangle=\langle(\Omega(H,\overline{\chi})\chi)^o\vert H^o\rangle-\langle(\Omega(\chi,\overline{H})\chi)^o\vert H^o\rangle$;

(12) $\Gamma_{\gamma;\alpha}^\beta=\hat{\Gamma}_{\gamma;\alpha}^\beta+i\hat{\Gamma}_{\gamma;\alpha}^{\beta^\ast  }$;

(13)  $\nabla_X \pmb{J}=0$ for any $X\in\mathcal{H}_{\mathbb{R}}$;

(14) $\langle\hat{\Omega}(\hat{\chi},\widehat{\pmb{J}\chi})\widehat{\pmb{J}\chi}\vert \hat{\chi}\rangle=2\langle\Omega(\chi,\overline{\chi})\chi,\chi\rangle$.
 \end{theorem}

\begin{proof}
The assertion (1) follows from Proposition 2.6.2 in \cite{AP}.

The assertion (5) follows from Theorem 2.6.4 in \cite{AP}, which together with Proposition  2.6.3 in \cite{AP} implies the assertions (2)-(4) and (6)-(8).

The assertion (9) follows from Theorem 2.6.6 in \cite{AP}.

The assertion (10) follows from Theorem 2.6.8 in \cite{AP} since by assertion (5), we have $\hat{H}=H^o$.

The assertion (11) follows from Theorem 2.6.9 in \cite{AP}. Since for a K\"ahler-Berwald metric we have $\frac{\partial\Gamma_{;\mu}^\alpha}{\partial \bar{v}^\gamma}=0$, hence $\left(\nabla_\chi\tau^{\mathcal{H}}\right)(\chi,\overline{W})=0$. Thus $\left(\left(\nabla_\chi\tau^{\mathcal{H}}\right)(\chi,\overline{\Theta^{-1}(H)})\right)^o=0$ for any $H\in\mathcal{H}^{1,0}$.

The assertion (12) follows from Theorem \ref{Th-ab}.

The assertion (13) follows from Theorem \ref{Th-d}.

Substituting $H=\pmb{J}\chi=i\chi$ into assertion (11), we get $$\langle\hat{\Omega}(\hat{\chi},\widehat{\pmb{J}\chi})\widehat{\pmb{J}\chi}\vert \hat{\chi}\rangle=\langle(\Omega(i\chi,\overline{\chi})\chi)^o\vert (i\chi)^o\rangle-\langle(\Omega(\chi,\overline{i\chi})\chi)^o\vert (i\chi)^o\rangle.$$
This together with \eqref{VW-a}
%, and the following equalities
%$$
%\langle\langle\Omega(i\chi,\overline{\chi})\chi,i\chi\rangle\rangle=0,\quad\langle\langle\Omega(\chi,\overline{i\chi})\chi,i\chi\rangle\rangle=0
%$$
implies
  \begin{eqnarray*}
  \langle\hat{\Omega}(\hat{\chi},\widehat{\pmb{J}\chi})\widehat{\pmb{J}\chi}\vert \hat{\chi}\rangle&=&\text{Re}[\langle\Omega(i\chi,\overline{\chi})\chi,i\chi\rangle-\langle \Omega(\chi,\overline{i\chi})\chi,i\chi\rangle]\\
  &=&\text{Re}[\langle\Omega(\chi,\overline{\chi})\chi,\chi\rangle+\langle \Omega(\chi,\overline{\chi})\chi,\chi\rangle]\\
  &=&2\langle\Omega(\chi,\overline{\chi})\chi,\chi\rangle
  \end{eqnarray*}
where we use the fact that $\langle\!\langle \cdot,\chi\rangle\!\rangle=0$ and $\langle\Omega(\chi,\overline{\chi})\chi,\chi\rangle$ is real-valued.
\end{proof}

\begin{proposition}\label{FF}
Let $\Phi=-iG_{\alpha\bar{\beta}}dz^{\alpha}\wedge d\bar{z}^{\beta}$ be the fundamental form of a strongly convex complex Finsler metric $F:T^{1,0}M\rightarrow[0,+\infty)$ on a complex manifold $M$. Then $\Phi$ is a real horizontal $(1,1)$-form on $\tilde{M}$. If moreover, $\nabla_X\pmb{J}=0$ for any $X\in\mathcal{H}_{\mathbb{R}}$, then
\begin{equation}
\Phi(X,Y)=\langle X\vert \pmb{J}Y\rangle -\langle \pmb{J}X\vert Y\rangle
\end{equation}
for any $X,Y\in\mathcal{H}_{\mathbb{R}}$.
\end{proposition}
\begin{proof}
It is clear that $\Phi$ is real $(1,1)$-form on $\tilde{M}$. If $\nabla_X  \pmb{J}\equiv 0$ for all $X\in\mathcal{H}_{\mathbb{R}}$, then $\mathcal{H}_{\mathbb{R}}$ is $\pmb{J}$-invariant and by assertion (5) in Theorem \ref{SE}, we have $H=(H_o)^o$ for any $H\in\mathcal{H}_{\mathbb{R}}$, namely $\mathcal{H}_{\mathbb{R}}\cong\mathcal{H}^{1,0}$. Thus
\begin{eqnarray*}
\Phi(X,Y)
&=&-iG_{\alpha\bar{\beta}}dz^\alpha\wedge d\bar{z}^\beta(V+\overline{V},W+\overline{W})\\
&=&-i\left(G _{\alpha\bar{\beta}}V^\alpha\overline{W}^\beta-G_{\alpha\bar{\beta}}W^\alpha\overline{V}^\beta\right)\\
&=&-i\left[\langle V,W\rangle-\overline{\langle V,W}\rangle\right]\\
&=&2\text{Im}\langle V,W\rangle.
\end{eqnarray*}
for $V=\frac{1}{2}(X-i\pmb{J}X)=X_o$ and $W=\frac{1}{2}(Y-i\pmb{J}Y)=Y_o\in \mathcal{H}^{1,0}$ with $X,Y\in \mathcal{H}_{\mathbb{R}}$.

Since by Lemma \ref{L-a}, we also have
\begin{eqnarray*}
\langle V,W\rangle=2\langle V\vert \overline{W}\rangle=\frac{1}{2}\left\{\langle X\vert Y\rangle+\langle JX\vert JY\rangle+i\left[\langle X\vert \pmb{J}Y\rangle -\langle \pmb{J}X\vert Y\rangle\right]\right\},
\end{eqnarray*}
this completes the proof.
\end{proof}

For a strongly convex complex Finsler manifold $(M,F)$, denote $d:\wedge^k\tilde{M}\rightarrow\wedge^{k+1}\tilde{M}$ the exterior derivative operator on $\tilde{M}$. Then
\begin{eqnarray}
d=d_H+d_V,\quad d_H=dx^j\wedge\nabla_{\frac{\delta}{\delta x^j}},\quad d_V=\delta y^j\wedge\nabla_{\frac{\partial}{\partial y^j}},\label{dh}
\end{eqnarray}
where $\nabla$ is the Cartan connection of $F$. If we denote
\begin{equation}
 \partial_H=dz^\alpha\wedge D_{\delta_\alpha},\; \bar{\partial}_H=d\bar{z}^\alpha\wedge D_{\delta_{\bar{\alpha}}},\; \partial_V=\psi^\gamma\wedge D_{\frac{\partial}{\partial v^\gamma}},\; \bar{\partial}_V=\bar{\psi}^\gamma\wedge D_{\frac{\partial}{\partial\bar{v}^\gamma}},\label{dv}
\end{equation}
where $D$  is the Chern-Finsler connection of $F$, then in general $d_H\neq \partial_H+\bar{\partial}_H$ since in general the Cartan connection $\nabla$ may not coincides with the Chern-Finsler connection $D$.
\begin{proposition}Let $d_H,d_V,\partial_H,\bar{\partial}_H,\partial_V,\bar{\partial}_V$ be defined by \eqref{dh} and \eqref{dv}, respectively.
If $\nabla_X\pmb{J}=0$ for all $X\in\mathcal{H}_{\mathbb{R}}$, then
\begin{equation}
d_H=\partial_H+\bar{\partial}_H,\quad d_V=\partial_V+\bar{\partial}_V.
\end{equation}
\end{proposition}
\begin{proof}
By assumption, $F$ must be a K\"ahler-Berwald metric, hence $\mathcal{H}_{\mathbb{R}}$ is $\pmb{J}$-invariant and $\widetilde{\mathcal{H}}^{1,0}=\mathcal{H}^{1,0}$.
\end{proof}

\begin{theorem}\label{Th-3.8}
Let $F: T^{1,0}M\rightarrow [0,+\infty)$ be a strongly convex complex Finsler metric on a complex manifold $M$.  If  $\nabla_X  \pmb{J}=0$ for all $X\in\mathcal{H}_{\mathbb{R}}$, then
\begin{enumerate}
\item[(1)] $\nabla_X\varPhi= 0$  for all $X\in\mathcal{H}_{\mathbb{R}}$;

\item[(2)] $d_H\varPhi= 0$, i.e., $\varPhi$ is $d_H$-closed.
\end{enumerate}
\end{theorem}
\begin{proof}
	(1) Let $H,X,Y\in\mathcal{H}_{\mathbb{R}}$. Then by Proposition \ref{FF},
	\begin{eqnarray*}
		(\nabla_H\Phi)(X,Y)&=&H\Phi(X,Y)-\Phi(\nabla_HX,Y)-\Phi(X,\nabla_HY)\\
		&=&H\langle X\vert  \pmb{J}Y\rangle-H\langle  \pmb{J}X\vert Y\rangle-\langle \nabla_HX\vert  \pmb{J}Y\rangle+\langle  \pmb{J}(\nabla_HX)\vert Y\rangle\\
		&&-\langle X\vert  \pmb{J}(\nabla_HY)\rangle+\langle  \pmb{J}X\vert \nabla_HY\rangle.
	\end{eqnarray*}
Since the Cartan connection is horizontal metric, namely $H\langle X\vert Y\rangle=\langle \nabla_HX\vert Y\rangle+\langle X\vert \nabla_HY\rangle$, the above equality becomes
$$
(\nabla_H\Phi)(X,Y)=\langle X\vert \nabla_H( \pmb{J}Y)\rangle-\langle \nabla_H( \pmb{J}X)\vert Y\rangle+\langle  \pmb{J}(\nabla_HX)\vert Y\rangle-\langle X\vert  \pmb{J}(\nabla_HY)\rangle.
$$
 Since $\nabla_H \pmb{J}=0$, we have
	\begin{eqnarray*}
		&&\langle X\vert \nabla_H( \pmb{J}Y)\rangle-\langle \nabla_H( \pmb{J}X\vert Y\rangle+\langle \pmb{J}(\nabla_HX)\vert Y\rangle-\langle X\vert  \pmb{J}(\nabla_HY)\rangle\\
		&=&\langle X\vert  \pmb{J}(\nabla_HY)\rangle-\langle  \pmb{J}(\nabla_HX)\vert Y\rangle+\langle  \pmb{J}(\nabla_HX)\vert Y\rangle-\langle X\vert  \pmb{J}(\nabla_HY)\rangle\\
		&=&0.
	\end{eqnarray*}
	Therefore $\nabla_H\Phi=0$ for any $H\in\mathcal{H}_{\mathbb{R}}$.
	
	(2) Let $X_1,X_2,X_3\in\mathcal{H}_{\mathbb{R}}$. Then
	\begin{eqnarray*}
		(d_H\Phi)(X_1,X_2,X_3)&=&X_1(\Phi(X_2,X_3))-X_2(\Phi(X_1,X_3))+X_3(\Phi(X_1,X_2))\\
		&&-\Phi([X_1,X_2],X_3)+\Phi([X_1,X_3],X_2)-\Phi([X_2,X_3],X_1).
	\end{eqnarray*}
Using $\nabla_{X_j}\Phi=0$, this becomes
\begin{eqnarray*}
(d_H\Phi)(X_1,X_2,X_3)&=&\Phi(\nabla_{X_1}X_2-\nabla_{X_2}X_1,X_3)
	-\Phi(\nabla_{X_1}X_3-\nabla_{X_3}X_1,X_2)\\
	&&+\Phi(X_1,\nabla_{X_3}X_2-\nabla_{X_2}X_3)\\
		&&-\Phi([X_1,X_2],X_3)+\Phi([X_1,X_3],X_2)-\Phi([X_2,X_3],X_1).
\end{eqnarray*}
	Since  $\theta(X_j,X_k)=\nabla_{X_j}X_k-\nabla_{X_k}X_j-[X_j,X_k]$, and $\theta(X_j,X_k)\in\mathcal{V}_{\mathbb{R}}$, while $\Phi$ is horizontal, we get $(d_H\Phi)(X_1,X_2,X_3)=0$. Hence, $d_H\Phi=0$.
\end{proof}

\begin{remark}
The operator $d_H$ is not the same one defined on page 95 in \cite{AP}, where $d_H\Phi=0$ is equivalent to $F$ being a K\"ahler-Finsler metric. Thus converse of Theorem \ref{Th-3.8} may not be true. So  far in literature, however, explicit examples of K\"ahler-Finsler metrics \cite{Zhong-b,Lin-Zhong,Ge-Zhong, Cao-Ge-Zhong, Zhong-c}  were  all proved to be K\"ahler-Berwald metrics. It is still open whether there exists an example of K\"ahler-Finsler metric which is not a K\"ahler-Berwald metric.
\end{remark}

\begin{theorem}\label{Th-3.9}
Let $F: T^{1,0}M\rightarrow [0,+\infty)$ be a strongly convex weakly K\"ahler-Finsler metric on a complex manifold $M$ and $\sigma:[0,1]\rightarrow M$ a smooth regular curve in $M$. Then the following assertions are equivalent:
\begin{enumerate}
\item[(1)]  $F$ is a strongly convex K\"ahler-Berwald metric;

\item[(2)] The types of complexified  vectors in $T_{\mathbb{R}}M$ are preserved under parallel transport along $\sigma$ with respect to  $\nabla$;

\item[(3)] For any parallel real vector field $V$ along $\sigma$ with respect to  $\nabla$,  $ JV$ is also parallel along $\sigma$  with respect to $\nabla$;

\item[(4)] $\nabla_X  \pmb{J}\equiv 0$ for all $X\in\mathcal{H}_{\mathbb{R}}$;

\end{enumerate}
\end{theorem}

\begin{proof}

(1)$\Rightarrow$(2): It suffices to show that a type $(1,0)$ complexified tangent vector is preserved under parallel transport along $\sigma$ with respect to $\nabla$. Let $\xi^{1,0}(0)$  a complexified tangent vector of type $(1,0)$ at $\sigma(0)$, namely there exists a real tangent vector $\xi_0$ at $\sigma(0)$ such that $\xi^{1,0}(0)=\frac{1}{2}(\xi_0-iJ\xi_0)$.
Let $\zeta$ be the parallel transport of $\xi^{1,0}(0)$ along $\sigma$ with respect to $\nabla$. Write $\zeta=\zeta^{1,0}+\zeta^{0,1}$, where $\zeta^{1,0}$ and $\zeta^{0,1}$ denote the $(1,0)$ and $(0,1)$ parts of $\zeta$, respectively. It suffices  to show that $\zeta^{1,0}=\zeta$ and $\zeta^{1,0}(0)=\xi^{1,0}(0)$, namely $\zeta^{0,1}\equiv 0$ along $\sigma$.

For this purpose, let $V=V^j(t)\frac{\partial}{\partial x^j}$ be the parallel transport of $\xi_0$ along $\sigma$ with respect to $\nabla$ such that $V(0)=\xi_0$, namely
\begin{equation}
\frac{dV^k}{dt}+V^l\hat{\Gamma}_{l;j}^k(\sigma(t);\dot{\sigma}(t))\dot{\sigma}^j=0.\label{NNN-a}
\end{equation}

Setting $JV=U^k(t)\frac{\partial}{\partial x^k}$. In the following we shall show that $JV$ is also parallel along $\sigma$ with respect to $\nabla$.  It is clear that
\begin{equation}
U^\beta(t)=-V^{\beta^\ast}(t),\quad U^{\beta^\ast}(t)=V^\beta(t).\label{NNN-b}
\end{equation}
By assumption, $F$ is a strongly convex K\"ahler-Berwald metric. Thus by Theorem \ref{Th-d}, $\nabla_X\pmb{J}=0$ for all $X\in\mathcal{H}_{\mathbb{R}}$. Hence by Theorem \ref{Th-ab}, the horizontal Cartan connection coefficients $\hat{\Gamma}_{j;k}^l$ satisfy \eqref{eq-ab}, namely
%Thus by the proof of Theorem \ref{Th-d}, the real non-linear connection coefficients $\hat{\Gamma}_j^k$ of $F$  satisfy
%\begin{equation}
%\hat{\Gamma}_\alpha^\mu=\hat{\Gamma}_{\alpha^\ast}^{\mu^\ast},\quad\hat{\Gamma}_\alpha^{\mu^\ast}=-\hat{\Gamma}_{\alpha^\ast}^\mu.\label{NN-a}
%\end{equation}
%By Theorem 1.2 in \cite{Zhong-a},  $F$ is also a real Berwald metric, thus the real non-linear connection $\tilde{\nabla}$ associated to $F$ is actually a real linear connection and the real Berwald connection coefficients $\hat{\mathbb{G}}^k_{jl}$ coincide with the horizontal Cartan connection coefficients $\hat{\Gamma}_{j;l}^k$. Thus if we differentiating \eqref{NN-a} with respect to $y^j$, we have
\begin{equation}
\hat{\Gamma}_{\alpha^\ast;j}^\beta=-\hat{\Gamma}_{\alpha;j}^{\beta^\ast},\quad\hat{\Gamma}_{\alpha^\ast;j}^{\beta^\ast}=\hat{\Gamma}_{\alpha;j}^\beta.\label{NNN-c}
\end{equation}
Note that
 $$
\frac{\pmb{D}U^k}{dt}=\frac{dU^k}{dt}+U^l\hat{\Gamma}_{l;j}^k(\sigma(t);\dot{\sigma}(t))\dot{\sigma}^j(t).
$$
This together with \eqref{NNN-a}-\eqref{NNN-c} implies that
\begin{eqnarray*}
\frac{\pmb{D}U^\beta}{dt}&=&\frac{dU^\beta}{dt}+U^l\hat{\Gamma}_{l;j}^\beta(\sigma(t);\dot{\sigma}(t))\dot{\sigma}^j(t)\\
&=&-\frac{dV^{\beta^\ast}}{dt}-V^{\alpha^\ast}\hat{\Gamma}_{\alpha^\ast;j}^{\beta^\ast}(\sigma(t);\dot{\sigma}(t))\dot{\sigma}^j(t)-V^\alpha\hat{\Gamma}_{\alpha;j}^{\beta^\ast}\dot{\sigma}^j(t)\\
&=&-\frac{dV^{\beta^\ast}}{dt}-V^l\hat{\Gamma}_{l;j}^{\beta^\ast}(\sigma(t);\dot{\sigma}(t))\dot{\sigma}^j(t)=0.
\end{eqnarray*}
By similar reason, we have $\frac{\pmb{D}U^{\beta^\ast}}{dt}=0$. Thus $JV$ is also parallel along $\sigma$ with respect to $\nabla$, namely
\begin{equation}
\frac{\pmb{D}U^k}{dt}=\frac{dU^k}{dt}+U^l\hat{\Gamma}_{l;j}^k(\sigma(t);\dot{\sigma}(t))\dot{\sigma}^j(t)=0.\label{NNN-d}
\end{equation}
So that $W=\frac{1}{2}(V-iJV)$ is parallel along $\sigma$ with respect to $\nabla$. Since $W(0)=\frac{1}{2}(V(0)-iJV(0))=\frac{1}{2}(\xi_0-iJ\xi_0)=\zeta^{1,0}(0)$. Thus $\zeta\equiv W$, namely $\zeta^{1,0}\equiv0$ along $\sigma$.

(2) $\Rightarrow$ (3): Let $V$ be a real vector field which is parallel along $\sigma$ with respect to $\nabla$. Write $V=V^{1,0}+V^{0,1}$, where $V^{1,0}=\frac{1}{2}(V-iJV)$ and $V^{0,1}=\frac{1}{2}(V+iJV)$. Let $W_1$ and $W_2$ be the parallel transport of $V^{1,0}(0)$ and $V^{0,1}(0)$ along $\sigma$ with respect to $\nabla$. By assertion (2),
$W_1$ and $W_2$ are vector fields of type  $(1,0)$ and $(0,1)$ along $\sigma$, respectively, and $W=W_1+W_2$ is parallel along $\sigma$ with respect to $\nabla$. Since $W(0)=W_1(0)+W_2(0)=V^{1,0}(0)+V^{0,1}(0)=V(0)$ so that $W\equiv V$, $W_1=V^{1,0}$ and $W_2=V^{0,1}$ along $\sigma$. Thus $JV=J(V^{1,0}+V^{0,1})=i(V^{1,0}-V^{0,1})$ is also parallel along $\sigma$ with respect to $\nabla$.

(3) $\Rightarrow$ (4): Suppose both $V=V^k(t)\frac{\partial}{\partial x^k}$ and $JV=U^k(t)\frac{\partial}{\partial x^k}$ are  parallel along $\sigma$ with respect to $\nabla$. Then by \eqref{NNN-a}
we have
\begin{eqnarray}
\frac{dV^\beta}{dt}+V^\alpha\hat{\Gamma}_{\alpha;j}^\beta(\sigma(t);\dot{\sigma}(t))\dot{\sigma}^j+V^{\alpha^\ast}\hat{\Gamma}_{\alpha^{\ast};j}^\beta(\sigma(t);\dot{\sigma}(t))\dot{\sigma}^j&=&0,\label{RR-a}\\
\frac{dV^{\beta^\ast}}{dt}+V^\alpha\hat{\Gamma}_{\alpha;j}^{\beta^\ast}(\sigma(t);\dot{\sigma}(t))\dot{\sigma}^j+V^{\alpha^\ast}\hat{\Gamma}_{\alpha^{\ast};j}^{\beta^\ast}(\sigma(t);\dot{\sigma}(t))\dot{\sigma}^j&=&0.\label{RR-b}
\end{eqnarray}
Similarly by \eqref{NNN-b} and \eqref{NNN-d}, we have
%\begin{eqnarray*}
%\frac{dU^\beta}{dt}+U^\alpha\hat{\Gamma}_{\alpha;j}^\beta(\sigma(t);\dot{\sigma}(t))\dot{\sigma}^j+U^{\alpha^\ast}\hat{\Gamma}_{\alpha^{\ast};j}^\beta(\sigma(t);\dot{\sigma}(t))\dot{\sigma}^j&=&0,\\
%\frac{dU^{\beta^\ast}}{dt}+U^\alpha\hat{\Gamma}_{\alpha;j}^k(\sigma(t);\dot{\sigma}(t))\dot{\sigma}^j+U^{\alpha^\ast}\hat{\Gamma}_{\alpha^{\ast};j}^k(\sigma(t);\dot{\sigma}(t))\dot{\sigma}^j&=&0.
%\end{eqnarray*}
%Using \eqref{NNN-b}, we have
\begin{eqnarray}
-\frac{dV^{\beta^\ast}}{dt}-V^{\alpha^\ast}\hat{\Gamma}_{\alpha;j}^\beta(\sigma(t);\dot{\sigma}(t))\dot{\sigma}^j+V^{\alpha}\hat{\Gamma}_{\alpha^{\ast};j}^\beta(\sigma(t);\dot{\sigma}(t))\dot{\sigma}^j&=&0,\label{RR-c}\\
\frac{dV^{\beta}}{dt}-V^{\alpha^\ast}\hat{\Gamma}_{\alpha;j}^{\beta^\ast}(\sigma(t);\dot{\sigma}(t))\dot{\sigma}^j+V^{\alpha}\hat{\Gamma}_{\alpha^{\ast};j}^{\beta^\ast}(\sigma(t);\dot{\sigma}(t))\dot{\sigma}^j&=&0.\label{RR-d}
\end{eqnarray}
By \eqref{RR-a} and \eqref{RR-d}, we have
\begin{equation}
V^\alpha(\hat{\Gamma}_{\alpha;j}^\beta-\hat{\Gamma}_{\alpha^{\ast};j}^{\beta^\ast})\dot{\sigma}^j+V^{\alpha^\ast}(\hat{\Gamma}_{\alpha^{\ast};j}^\beta+\hat{\Gamma}_{\alpha;j}^{\beta^\ast})\dot{\sigma}^j=0.\label{RR-e}
\end{equation}
By \eqref{RR-b} and \eqref{RR-c}, we have
\begin{equation}
V^\alpha(\hat{\Gamma}_{\alpha;j}^{\beta^\ast}+\hat{\Gamma}_{\alpha^{\ast};j}^\beta)\dot{\sigma}^j+V^{\alpha^\ast}(\hat{\Gamma}_{\alpha^{\ast};j}^{\beta^\ast}-\hat{\Gamma}_{\alpha;j}^\beta)\dot{\sigma}^j=0.\label{RR-f}
\end{equation}
Since $\hat{\Gamma}_{j;l}^k=\hat{\Gamma}_{l;j}^k$ and $\hat{\Gamma}_{j;l}^k(x;y)y^j=\hat{\Gamma}_l^k$, it follows that \eqref{RR-e} and \eqref{RR-f} can be simplified as
\begin{eqnarray}
V^\alpha(\hat{\Gamma}_{\alpha}^\beta-\hat{\Gamma}_{\alpha^{\ast}}^{\beta^\ast})+V^{\alpha^\ast}(\hat{\Gamma}_{\alpha^{\ast}}^\beta+\hat{\Gamma}_{\alpha}^{\beta^\ast})&=&0,\label{RR-g}\\
V^\alpha(\hat{\Gamma}_{\alpha}^{\beta^\ast}+\hat{\Gamma}_{\alpha^{\ast}}^\beta)+V^{\alpha^\ast}(\hat{\Gamma}_{\alpha^{\ast}}^{\beta^\ast}-\hat{\Gamma}_{\alpha}^\beta)&=&0.\label{RR-h}
\end{eqnarray}
Since $\hat{\Gamma}_j^k(\sigma;\dot{\sigma})$ are independent of $V^j$ for $j=1,\cdots,2n$, $\sigma:[0,1]\rightarrow M$ is an arbitrary fixed smooth regular curve in $M$ and $V=V^j(t)\frac{\partial}{\partial x^j}$ is an arbitrary parallel vector field along $\sigma$ with respect to $\nabla$, it follows that \eqref{RR-g} and \eqref{RR-h} hold iff
\begin{equation}
\hat{\Gamma}_{\alpha}^\beta(x;y)=\hat{\Gamma}_{\alpha^{\ast}}^{\beta^\ast}(x;y),\quad \hat{\Gamma}_{\alpha^{\ast}}^\beta(x;y)=-\hat{\Gamma}_{\alpha}^{\beta^\ast}(x;y),\quad \forall (x;y)\in\tilde{M}.\label{RR-i}
\end{equation}
%Since $F$ is a strongly convex weakly K\"ahler-Finsler metric, the real spray coefficients $\hat{\mathbb{G}}^k$ and the complex spray coefficients $\mathbb{G}^\mu$ associated to $F$ are related by
%$$\mathbb{G}^\mu=\hat{\mathbb{G}}^{\mu}+i\hat{\mathbb{G}}^{\mu^\ast},\quad \mu=1,\cdots,n.$$

%Note that $\mathbb{G}^\mu$ are $(2,0)$-homogeneous, namely $\mathbb{G}(\lambda v)=\lambda^2\mathbb{G}(v)$ for any $\lambda\in\mathbb{C}\setminus\{0\}$ and $v=y_o\in \tilde{M}$. It follows that
% $$\mathbb{G}^\mu(Jv)=-\mathbb{G}^\mu(v),$$
%  hence
%  $$\hat{\mathbb{G}}^{\mu}(Jy)=-\hat{\mathbb{G}}^{\mu}(y),\quad \hat{\mathbb{G}}^{\mu^\ast}(Jy)=-\hat{\mathbb{G}}^{\mu^\ast}(y).$$

%Thus \begin{eqnarray*}
%	\hat{\Gamma}^\mu_\alpha(J\xi)&=&\frac{\partial\hat{\mathbb{G}}^{\mu}(J\xi)}{\partial u^\alpha}=-\frac{\partial\hat{\mathbb{G}}^{\mu}(\xi)}{\partial u^\alpha}=	\hat{\Gamma}^\mu_{\alpha^\ast}(\xi),\\ 	%\hat{\Gamma}^{\mu^\ast}_\alpha(J\xi)&=&\frac{\partial\hat{\mathbb{G}}^{\mu^\ast}(J\xi)}{\partial u^\alpha}=-\frac{\partial\hat{\mathbb{G}}^{\mu^\ast}(\xi)}{\partial u^\alpha}=	\hat{\Gamma}^{\mu^\ast}_{\alpha^\ast}(\xi),
%\end{eqnarray*}
%where we used the equality $\frac{\partial}{\partial u^\alpha}=-\frac{\partial}{\partial y^{\alpha^\ast}}$.
%Comparing with  \eqref{JN-d}, we obtain the following equalities:
% \begin{equation}
% \hat{\Gamma}_{\alpha^\ast   }^\mu(\xi)=-\hat{\Gamma}_{\alpha}^{\mu^\ast  }(\xi),\quad
%\hat{\Gamma}_{\alpha^\ast   }^{\mu^\ast  }(\xi)=\hat{\Gamma}_{\alpha}^\mu(\xi). \label{JN-e}
%\end{equation}
Differentiating \eqref{RR-i} with respect to $y^j$ yields
$$\hat{\mathbb{G}}_{\alpha^\ast j  }^\mu(x;y)=-\hat{\mathbb{G}}_{\alpha j}^{\mu^\ast  }(x;y),\quad
\hat{\mathbb{G}}_{\alpha^\ast j }^{\mu^\ast  }(x;y)=\hat{\mathbb{G}}_{\alpha j}^\mu(x;y). $$
 This implies that $\pmb{J}$ is horizontal parallel with respect to the real Berwald  connection $\breve{\nabla}$ associated to $F$.

Since $F$ is a strongly convex weakly K\"ahler-Finsler metric, by Lemma 3.1 in \cite{XZ}, the real spray coefficients $\hat{\mathbb{G}}^k$ and the complex spray coefficients $\mathbb{G}^\mu$ satisfy
\begin{equation}
\mathbb{G}^\mu=\hat{\mathbb{G}}^\mu+i\hat{\mathbb{G}}^{\mathbb{\mu}^\ast}.\label{RR-j}
\end{equation}
Differentiating \eqref{RR-j} with respect to $\bar{v}^\alpha$ and using \eqref{RR-i}, we obtain
\begin{equation}
\frac{\partial\mathbb{G}^\mu}{\partial\bar{v}^\alpha}=\frac{1}{2}[(\hat{\Gamma}_\alpha^\mu-\hat{\Gamma}_{\alpha^\ast}^{\mu^\ast})+i(\hat{\Gamma}_\alpha^{\mu^\ast}+\hat{\Gamma}_{\alpha^\ast}^\mu)]=0.
\end{equation}
This implies that $\mathbb{G}^\mu$ are locally holomorphic with respect to the  fiber coordinates $v$. Note that $\mathbb{G}^\mu$ are $(2,0)$-homogeneous with respect to the local holomorphic coordinates $v$, namely $\mathbb{G}^\mu(z;\lambda v)=\lambda^2\mathbb{G}^\mu(z;v)$. Thus $\mathbb{G}^\mu$ must actually be quadratic with respect to $v$, that is,
$$
\mathbb{G}^\mu=\frac{1}{2}\mathbb{G}_{\alpha\gamma}^\mu(z)v^\alpha v^\gamma,
$$
where $\mathbb{G}_{\alpha\gamma}^\mu(z;v)=\mathbb{G}_{\alpha\gamma}^\mu(z)$ are the complex Berwald connection coefficients of $F$. This shows that $F$ is a weakly complex Berwald metric \cite{Zhong-a}. Since $F$ is a strongly convex weakly K\"ahler-Finsler metric, thus by Theorem 1.1 in \cite{Zhong-a}, $F$ is also a real Berwald metric. Hence $\hat{\mathbb{G}}^k_{jl}=\hat{\Gamma}^k_{j;l}$, namely the real Berwald connection coefficients $\hat{\mathbb{G}}_{jl}^k$ coincide with the horizontal Cartan connection coefficients $\hat{\Gamma}_{j;l}^k$. Thus  the condition \eqref{eq-ab} in Theorem \ref{Th-ab} are satisfied, hence $\nabla_X  \pmb{J}\equiv 0$ for all $X\in\mathcal{H}_{\mathbb{R}}$.

(4)$\Rightarrow$ (1): Theorem \ref{Th-c}.
\end{proof}

\section{K\"ahler-Berwald metrics of constant holomorphic sectional curvature}

In this section, we are able to use the geometric assumption $\nabla \pmb{J}=0$ to classify all strongly convex complex Finsler metrics with constant holomorphic sectional curvatures.

\begin{theorem}\label{RKE}
Let $F: T^{1,0}M\rightarrow [0,+\infty)$ be a strongly convex complex Finsler metric on a complex manifold $M$.  If $\nabla \pmb{J}=0$ and $F$ has constant holomorphic sectional curvature $c\neq 0$, then $F$ is necessary a K\"ahler-Einstein metric on $M$.
\end{theorem}
\begin{proof}
Using Definition 2.5.2 in \cite{AP}, it follows that $F$ has constant holomorphic sectional curvature $c$ along a nonzero tangent vector $v \in T_z^{1,0}M$ iff
	\begin{equation}
-2G_\alpha\frac{\partial \Gamma^\alpha_{;\mu}}{\partial \bar{z}^{\nu}}v^{\mu} \overline{v^\nu}=cG^2.\label{chc-1}
	\end{equation}
 Setting $A^\alpha_{\beta;\mu\bar{\nu}}:=-\frac{\partial\Gamma^{\alpha}_{\beta;\mu} }{\partial \bar{z}^\nu}$. By Theorem \ref{Th-c}, $F$ must be a K\"ahler-Berwald metric since by assumption $\nabla\pmb{J}=0$ implies particularly $\nabla_X\pmb{J}=0$ for any $X\in\mathcal{H}_{\mathbb{R}}$. So that we have $\Gamma_{\beta;\mu}^\alpha(z)=\Gamma_{\mu;\beta}^\alpha(z)$, hence
 $A_{\beta;\mu\bar{\nu}}^\alpha=A_{\mu;\beta\bar{\nu}}^\alpha$ and
 $\Gamma_{;\mu}^\alpha=\Gamma_{\beta;\mu}^\alpha(z)v^\beta$ are holomorphic and complex linear with respect to $v$. Differentiating \eqref{chc-1} with respect to $\bar{v}^\gamma$ yields
\begin{equation}\label{chc-2}
	G_{\alpha\bar{\gamma}}A^\alpha_{\beta;\mu\bar{\nu}} v^{\beta} v^{\mu} \overline{v^\nu}+G_{\alpha}A^\alpha_{\beta;\mu\bar{\gamma}} v^{\beta} v^{\mu} =cGG_{\bar{\gamma}}.
\end{equation}

 Rewritten \eqref{chc-2}, we obtain
 \begin{equation}
 \left\langle\Omega(\chi,\bar{\chi})\chi,\frac{\delta}{\delta z^\gamma} \right\rangle_v+\left\langle\Omega\left(\chi,\frac{\delta}{\delta \bar{z}^\gamma}\right)\chi,\chi \right\rangle_v=cGG_{\bar{\gamma}}.
 \label{chc-22}
 \end{equation}

Using Proposition 2.5.1 in \cite{AP}, we immediately have
\begin{eqnarray*}
\left\langle\Omega(\chi,\bar{\chi})\chi,\frac{\delta}{\delta z^\gamma} \right\rangle_v&=&\overline{\left\langle\Omega(\chi,\bar{\chi})\frac{\delta}{\delta z^\gamma},\chi \right\rangle_v},\\ \left\langle\Omega\left(\chi,\frac{\delta}{\delta \bar{z}^\gamma}\right)\chi,\chi \right\rangle_v&=&\overline {\left\langle\Omega\left(\frac{\delta}{\delta z^\gamma},\bar{\chi}\right)\chi,\chi \right\rangle_v}.
\end{eqnarray*}
Using the first equality of (2.3.21) in \cite{AP}, it follows that
\begin{eqnarray}\label{hs-1} \left\langle\Omega(\chi,\bar{\chi})\frac{\delta}{\delta z^\gamma},\chi \right\rangle_v&=&-G_\alpha\frac{\partial \Gamma^\alpha_{\gamma;\mu}}{\partial \bar{v}^\nu}v^\mu\overline{v^\nu}-G_\alpha\Gamma^{\alpha}_{\gamma\sigma}\frac{\partial\Gamma^\sigma_{;\mu}}{\partial \bar{v}^\nu}v^\mu\overline{v^\nu}\nonumber.
\end{eqnarray}

  By assumption $\nabla\pmb{J}=0$, it follows from Theorem \ref{Th-cc} that the Cartan connection of $F$  coincides with the Chern-Finsler connection of $F$, hence the vertical Chern-Finsler connection coefficients $\Gamma_{\gamma\sigma}^\alpha$ and the vertical Cartan connection coefficients $\hat{\Gamma}_{bc}^a$ satisfy the equality  $\Gamma_{\gamma\sigma}^\alpha=\hat{\Gamma}^{\alpha}_{\gamma\sigma}+i\hat{\Gamma}^{\alpha^\ast}_{\gamma\sigma}$. On the other hand, the vertical Cartan connection coefficients $\hat{\Gamma}_{bc}^a$ of $F$ always satisfy $$G_a\hat{\Gamma}_{bc}^a=\frac{1}{4}G_ag^{as}\frac{\partial^3F^2}{\partial y^b\partial y^c\partial y^s}=\frac{1}{2}\frac{\partial^3F^2}{\partial y^b\partial y^c\partial y^s}y^s=0,$$
   hence
  \begin{eqnarray*}
	2G_\alpha\Gamma^{\alpha}_{\gamma\sigma}&=&(G_{\alpha}-iG_{\alpha^\ast})(\hat{\Gamma}^{\alpha}_{\gamma\sigma}+i\hat{\Gamma}^{\alpha^\ast}_{\gamma\sigma})\\
%	&=&G_{\alpha}\hat{\Gamma}^{\alpha}_{\gamma\sigma}-iG_{\alpha^\ast}\hat{\Gamma}^{\alpha}_{\gamma\sigma}+iG_{\alpha}\hat{\Gamma}^{\alpha^\ast}_{\gamma\sigma}+G_{\alpha^\ast}\hat{\Gamma}^{\alpha^\ast}_{\gamma\sigma}\\
	&=&(G_{\alpha}\hat{\Gamma}^{\alpha}_{\gamma\sigma}+G_{\alpha^\ast}\hat{\Gamma}^{\alpha^\ast}_{\gamma\sigma})-i(G_{\alpha^\ast}\hat{\Gamma}^{\alpha^\ast}_{\gamma^\ast\sigma}+G_{\alpha}\hat{\Gamma}^{\alpha}_{\gamma^\ast\sigma})\\
	&=&0,
	\end{eqnarray*}
Thus
\begin{eqnarray*} \left\langle\Omega(\chi,\bar{\chi})\frac{\delta}{\delta z^\gamma},\chi \right\rangle_v
	=-G_\alpha\frac{\partial\Gamma^\alpha_{\gamma;\mu}}{\partial \bar{v}^\nu}v^\mu\overline{v^\nu}
	=\left\langle\Omega\left(\frac{\delta}{\delta z^\gamma},\bar{\chi}\right)\chi,\chi \right\rangle_v
\end{eqnarray*}
since $\Gamma_{\gamma;\mu}^\alpha=\Gamma_{\mu;\gamma}^\alpha$.
Therefore we actually obtain
$$\left\langle\Omega(\chi,\bar{\chi})\chi,\frac{\delta}{\delta z^\gamma} \right\rangle_v=\left\langle\Omega\left(\chi,\frac{\delta}{\delta \bar{z}^\gamma}\right)\chi,\chi \right\rangle_v.$$
Thus \eqref{chc-22} can be simplified as
\begin{equation}\label{chc-4}
	2A^\alpha_{\beta;\mu\bar{\nu}} v^{\beta} v^{\mu} \overline{v^\nu} =cGv^{\alpha}.
\end{equation}
Differentiating \eqref{chc-4} with respect to $\bar{v}^\nu$ yields
\begin{equation}
	2A^\alpha_{\beta;\mu\bar{\nu}} v^{\beta} v^{\mu}  =cG_{\bar{\nu}}v^{\alpha}.\label{chc-aaa}
\end{equation}
Differentiating \eqref{chc-aaa} with respect to $v^\beta$ and $v^\mu$ successively yields
\begin{equation}\label{chc-5}
	4A^\alpha_{\beta;\mu\bar{\nu}}=c\left(G_{\beta\bar{\nu}}\delta^{\alpha}_{\mu}+cG_{\mu\bar{\nu}}\delta^{\alpha}_{\beta}+cG_{\beta\bar{\nu}\mu}v^{\alpha}\right).
\end{equation}
Taking $\alpha=\beta$ and then summing $\alpha$ from $1$ to $n$, we obtain
\begin{equation}\label{chc-6}
4\sum_{\alpha=1}^nA^\alpha_{\alpha;\mu\bar{\nu}}=c(n+1)G_{\mu\bar{\nu}}\quad\mbox{equivalently}\quad G_{\mu\bar{\nu}}=\frac{4}{c(n+1)}\text{Ric}_{\mu\bar{\nu}}.
	\end{equation}
This implies that $F$ is a K\"ahler-Einstein metric on $M$ (since $\text{Ric}_{\mu\bar{\nu}}$ are independent of $v$, hence $G_{\mu\bar{\nu}}$ of $F$ are also independent of the fiber coordinates $v$).
\end{proof}

Now we immediately obtain the following rigidity theorem.
 \begin{theorem}\label{RKB}
 	Let $F: T^{1,0}M\rightarrow[0,+\infty)$ be a complete strongly convex complex Finsler metric on a simply connected complex manifold $M$. If $\nabla\pmb{J}=0$ and $F$ has constant holomorphic sectional curvature $c\in \mathbb{R}$. Then
 \begin{itemize}
 \item[(i)] if $c<0$, $(M,F)$ is a K\"ahler manifold which is holomorphically isometric to the open unit ball $B_n$ in $\mathbb{C}^n$ with a constant multiple of the Bergman metric, namely
 \begin{equation}
 F^2(z;v)=-\frac{4}{c}\frac{(1-\| z\| ^2)\| v\| ^2+\vert \langle z,v\rangle\vert ^2}{(1-\| z\| ^2)^2}
 \end{equation}

\item[(ii)] if $c=0$, $(M,F)$ is locally a complex Minkowski space which is locally holomorphic isometric to $\mathbb{C}^n$ with some complex Minkowski metric $F(z;v)=f(v)$;
 	
 \item[(iii)] if $c>0$,  $(M,F)$ is a K\"ahler manifold which is holomorphically isometric to the complex projective space $\mathbb{CP}^n$ with a constant multiple of the Fubini-Study metric, namely
 \begin{equation}
 F^2(z;v)=\frac{4}{c}\frac{(1+\| z\| ^2)\| v\| ^2-\vert \langle z,v\rangle\vert ^2}{(1+\| z\| ^2)^2}.
 \end{equation}
\end{itemize}
\end{theorem}

\begin{proof}
For $c=0$, it follows from Theorem 2.1 in \cite{Aikou-b} that $F$ is locally a complex Minkowski metric.

For $c\neq 0$, it follows from Theorem \ref{RKE} that $F$ is a K\"ahler-Einsetin metric on $M$. Since $M$ is simply connected and $F$ has constant holomorphic sectional curvature $c$, the classical result of K\"ahler geometry (cf. page 171 in \cite{Kobayashi}) implies the assertions (i) and (iii).
\end{proof}

The nontrivial case of Theorem \ref{RKB} is to show the case whenever $M$ admits both Hermitian quadratic and non-Hermitian quadratic K\"ahler-Berwald metrics, the constant holomorphic sectional curvature property implies K\"ahler-Berwald metrics must actually be K\"ahler metrics. Indeed, there exist lots of complex manifolds which admit both Hermitian quadratic and non-Hermitian quadratic K\"ahler-Berwald metrics, such as the unit polydisks in $\mathbb{C}^n$ ($n\geq 2$) \cite{Zhong-b},\cite{Lin-Zhong} and irreducible bounded symmetric domains \cite{Ge-Zhong},\cite{Zhong-c}. There are also complex manifolds, such as the open unit ball $B_n$ in $\mathbb{C}^n$  and the complex projective space $\mathbb{CP}^n$ \cite{Cao-Ge-Zhong}, which admit no non-Hermitian quadratic K\"ahler-Berwald metrics other than constant scaling of the standard K\"ahler metrics on $B_n$ and $\mathbb{CP}^n$, respectively.

\begin{remark}\label{remark-4.1}
 In \cite{Aikou-b}, under the assumptions that $(M,F)$ is a simply connected and complete complex manifold modeled on a complex Minkowski space $(\mathbb{C}^n,f)$, that $F$ is a K\"ahler-Finsler metric and satisfies the Royden condition, and that $(M,F)$ has constant holomorphic sectional curvature $c$, Aikou outlined a proof of the above theorem (using the nontrivial assertion that "if a K\"ahler-Berwald manifold $(M,F)$ is of constant holomorphic sectional curvature $c$, then $(M,h_M)$ is also of constant holomorphic sectional curvature $c$" without proof). More precisely, the equality $F(z;v)=f(v)$ in assertion (ii) is replaced by $F(z;v)=f(v)\leq \sum_{k=1}^n\vert v^k\vert ^2$ and the "$=$" in assertion (iii) is replaced by "$\leq$" in \cite{Aikou-b}.  In contrast, our proof of Theorem \ref{RKB} is based on the rigidity Theorem \ref{RKE} which is essentially different from that of \cite{Aikou-b}.  Moreover, our hypothesis $\nabla\pmb{J}=0$ is a simple geometric condition.
\end{remark}

\begin{remark}The complex Minkowski metric $f(v)$ in assertion (ii) can not be replaced by the canonical Euclidean metric $f(v)=\sum_{\alpha=1}^n\vert v^\alpha\vert^2$. Indeed, for any fixed $t\in[0,+\infty)$ and integer $k\geq 2$,
 $$f_{t,k}(v)=\vert v^1\vert^2+\cdots+\vert v^n\vert^2+t\sqrt{\vert v^1\vert^{2k}+\cdots+\vert v^n\vert^{2k}},\quad \forall v=(v^1,\cdots,v^n)\in\mathbb{C}^n$$
 is a strongly convex complex Minkowski metric on $\mathbb{C}^n$ \cite{Zhong-b},  and $(\mathbb{C}^n,f_{t,k})$  is a complex Minkowski space with vanishing holomorphic sectional curvature.  However, $(\mathbb{C}^n,f_{t_1,k_1})$ is not holomorphic isometric to $(\mathbb{C}^n,f_{t_2,k_2})$ whenever $t_1\neq t_2$ and $k_1\neq k_2$.
\end{remark}

\begin{remark}
The method developed in this paper can also be used to study the parallelism of the canonical complex structure $\pmb{J}$ on the pull-back bundle $\pi^\ast T_{\mathbb{R}}M$ with respect to the Chern-Rund connection of $F$, this will appear elsewhere.
\end{remark}

\vskip0.4cm
\noindent
\textbf{Acknowledgement.}\ {\small The second author is supported by National Natural Science Foundation of China (Grant No. 12471080).

%\vskip0.4cm
%\noindent
%\textbf{Declarations}\\
%\noindent
%\textbf{Funding} (This work is partially supported by National Natural Science Foundation of China (Grant Nos. 12071386, 11671330)\\
%\noindent
%\textbf{Conflicts of interest/Competing interests} (There is no conflicts of interest)\\
%\noindent
%\textbf{Availability of data and material} (My manuscript has no associated data)\\
%\noindent
%\textbf{Code availability} (Not applicable)\\


\begin{thebibliography}{9}
\bibitem{AP} M. Abate and G. Patrizio, \emph{Finsler Metrics-A Global Approach with Applications to Geometric Function
Theory}, Lecture Notes in Mathematics, vol. 1591, Springer-Verlag,  Berlin, 1994.

\bibitem{Aikou-a}T. Aikou, On complex Finsler manifolds. Rep. Fac. Sci., Kogoshima Univ. (Math., Phys. $\&$ Chem.), No.24, p.9-25, 1991.
\bibitem{Aikou-b} T. Aikou, Complex manifolds modeled on a complex Minkowski space, \emph{J. Math. Kyoto Univ.} {\bf 35} (1995), no.~1, 85-103.

\bibitem{BCS}D. Bao, S.-S. Chern, and Z. Shen, \emph{An Introduction to Riemann-Finsler Geometry}, Graduate Text in Mathematics, vol. 200, Springer-Verlag, New York, 2000.
\bibitem{Cao-Ge-Zhong}P. Cao, X. Ge and C. Zhong, Characterization of invariant complex Finsler metrics on the complex Grassmann manifold, \emph{Differential Geom. Appl.} \text{94}(2024), Paper No. 102138, 22pp.

\bibitem{Chen-Shen}B. Chen and Y. Shen,  K\"ahler Finsler metrics are actually strongly K\"ahler, \emph{Chin. Ann. Math. B}\textbf{30}(2009), no.2, 173-178.
\bibitem{chen2} B. Chen, S. Liu, and L. Zhao, Ricci curvature of strongly convex K\"ahler-Finsler metrics, \emph{Sci. China Math.} (in Chinese),  \textbf{54}(2024), no. 10, 1469-1488.

\bibitem{Ge-Zhong}X. Ge and C. Zhong, Geometry of holomorphic invariant strongly pseudoconvex complex Finsler metrics on the classical domains, \emph{Sci. China Math.} \textbf{67}(2024), no.8, 1827-1864.
\bibitem{he-zhong} Y. He and  C. Zhong,  Strongly convex weakly complex Berwald metrics and real Landsberg metrics, \emph{Sci. China Math.} \textbf{61}(2018), 535-544 .
\bibitem{Kobayashi}S. Kobayashi,  K. Nomizu, \emph{Foundations of Differential Geometry}, Vol.2, Interscience, New York, 1969.

\bibitem{Lin-Zhong}S. Lin and C. Zhong, Characterization of holomorphic invariant strongly pseudoconvex complex Finsler metric on unit polydisks, \emph{J. Geom. Anal.} \textbf{33}(2023), no.11, Paper No. 343, 28pp.

\bibitem{Mok}N. Mok, \emph{Metric rigidity theorems on Hermitian locally symmetric manifolds}, World Scientific, Singapore, 1989.

\bibitem{Royden}H. L. Royden, Complex Finsler metrics, \emph{Contemporary Mathematics}, \textbf{49}(1986), 119-124.

\bibitem{Shen}Z. Shen, \emph{Differential geometry of spray and Finsler spaces}, Kluwer Academic Publishers, Dordrecht, 2001.

\bibitem{XZ}H. Xia and C. Zhong, On strongly convex weakly K\"ahler-Finsler metrics of constant flag curvature, \emph{J. Math. Anal. Appl.} \textbf{443}(2016), 891-912.
\bibitem{Zhong-a}C. Zhong, On real and complex Berwald connections associated to strongly convex weakly K\"ahler-Finsler metrics, \emph{Differential Geom. Appl.} \textbf{29}(2011), 388-408.

\bibitem{Zhong-b} C. Zhong, De Rham decomposition theorem for strongly convex K\"ahler-Berwald manifolds,\emph{Results Math.} \textbf{78}(2023), no.1, Paper No. 25, 47pp.

\bibitem{Zhong-c} C. Zhong,  Characterization of holomorphic invariant complex Finsler metrics and Schwarz lemma on the classical domains, \emph{Math. Ann.} \textbf{392}(2025), no.4, 4861-4890.


\end{thebibliography}

%\input sn-sample-bib.tex%

\end{document}